\documentclass[a4paper,12pt,reqno]{amsart}

\usepackage{amsmath}
\usepackage{amscd}
\usepackage{amssymb}
\usepackage{mathrsfs}   
\usepackage[left=2.5cm,right=2.5cm,bottom=3cm,top=3cm]{geometry}
\newtheorem{theorem}{Theorem}[section]

\newtheorem{lemma}[theorem]{Lemma}

\newtheorem{prop}[theorem]{Proposition}
\theoremstyle{definition}

\newtheorem{rem}[theorem]{Remark}

\numberwithin{equation}{section}
\newcommand{\refl}[1]{Lemma~{\rm \ref{#1}}}
\newcommand{\reft}[1]{Theorem~{\rm \ref{#1}}}
\newcommand{\refe}[1]{(\ref{#1})}

\newcommand{\dd}{\sqrt{-1}\partial\bar{\partial}}

\begin{document}

\title[Green's function estimates]{Green's function estimates for compact K\"ahler manifolds and applications}

\author{Weiqi Zhang}
\author{Yashan Zhang}

\thanks{This paper substantially expands an older version on arXiv (Green's function and volume noncollapsing estimates for the K\"ahler-Ricci flow, arXiv:2508.13646v1, August 2025), whose main results contain Green's function integral estimates in Theorems \ref{thmgreen} and \ref{thmgreen-q}, volume non-collapsing estimates in Theorem \ref{thmgeom} (a) and Proposition \ref{prop-vol-explicit} and their applications to the K\"ahler-Ricci flow and K\"ahler family, among others.}

\address{Hunan University, Changsha, China}

\email{weiqizhang@hnu.edu.cn}
\email{yashanzh@hnu.edu.cn}


\begin{abstract}
Recent works of Guo-Phong-Song-Sturm established for compact K\"ahler manifolds (even for K\"ahler spaces of specific singularities) a variety of geometric estimates depending on an upper bound of $L^{1+\epsilon}$ or $L^1(\log L)^{n+\epsilon}$ norms of the volume density but not on any curvature bound, in which a key ingredient is a uniform integral estimate for Green's function. Motivated by their results and further applications, in this paper we shall prove an improved (nearly optimal) integral estimate for Green's function under $L^{1+\epsilon}$ volume density condition, and then apply it to obtain improved global geometric estimates. For instance, one of our results states that the $k$th eigenvalue of Laplacian operator $\lambda_k\ge c\cdot k^{\frac{1}{n}}(\log k)^{-3}$, where $n$ is the complex dimension of the K\"ahler manifold and $c$ depends on $n$ and $L^{1+\epsilon}$ norm of the volume density. Also, our results can be applied to the long-time or volume-noncollapsing finite-time K\"ahler-Ricci flow on compact K\"ahler manifolds and to a general K\"ahler family to further extend previous works of Guo-Phong-Song-Sturm, Guedj-T\^o and Vu.
\end{abstract}

\maketitle
\setcounter{tocdepth}{1}
\tableofcontents
\section{Introduction}
\subsection{Backgrounds and motivations}
In Riemannian geometry, there are classical uniform geometric estimates under certain curvature conditions, particularly Ricci curvature lower bound, which play a fundamental role in the study of geometric and topological structure of Riemannian manifolds in the past decades. As one of the most fundamental ones, Bonnet-Myers theorem bounds the diameter from above in terms of a positive lower bound of the Ricci curvature. 

In the case of compact K\"ahler manifolds, as motivated by several important questions, a fundamental problem is to understand uniform geometric properties of a degenerate family of K\"ahler metrics with or without a uniform Ricci curvature lower bound, on which significant propresses have been obtained in recent years. The diameter bound depending on a Ricci curvature lower bound and $L^\infty$ norm of the K\"ahler potential has been built in Fu-Guo-Song \cite{FGS} and Guedj-Guenancia-Zeriahi \cite{GGZ}; more related works can be found in \cite{GS,LTZ,STZ,TWY,Zys,ZZ} and references therein. In the case of no Ricci curvature lower bound being assumed, thanks to Y. Li \cite{Li} one can bound the H\"older norm of the distance function in terms of the H\"older norm of the K\"ahler potential, which in turn can be bounded in terms of $L^{1+\epsilon}$ norm of the volume density and certain background geometric bounds, according to H\"older norm estimates for the K\"ahler potential in \cite{K08,D-etc}. Recent remarkable works of Guo-Phong-Song-Sturm \cite{GPSt,GPSS1,GPSS2,GPSS3} established a variety of uniform global  geometric estimates (including diameter) for compact K\"ahler manifolds (with possibly specific singularities), which do not depend on any curvature condition but depend on certain entropy upper bound and Tian's $\alpha$-invariant, and have been applied to achieve a series of striking geometric consequences. For convenience and reference, we summarize part of results in \cite{GPSt,GPSS1,GPSS2,GPSS3} as follows. First fix some notations. Throughout this paper, let $(X,\omega_X)$ be an $n$-dimensional compact K\"ahler manifold with $n\ge2$. For a K\"ahler metric $\omega$ on $X$, let $G_\omega$ and $H_\omega$ be Green's function and heat kernel of $(X,\omega)$ defined by
$$\Delta_\omega G_\omega(x,\cdot)=-\delta_x(\cdot)+V_\omega^{-1}\,\,\,\mathrm{and}\,\,\,\int_XG_\omega(x,\cdot) \omega^n=0,$$
and 
\begin{equation}
    \frac{\partial}{\partial t} H(x,y,t)=\Delta_{\omega,y}H(x,y,t)\,\,\, \mathrm{and}\,\,\, \lim_{t\to 0^+} H(x,y,t)=\delta_x(y)\nonumber,
\end{equation} 
respectively, where $V_\omega:=Vol(X,\omega)$, $\Delta_\omega$ is the Laplacian operator associated to $\omega$ and $x,y\in X$; write 
$$0=\lambda_0<\lambda_1\le\lambda_2\le\cdots$$
for the increasing sequence of eigenvalues of $-\Delta_\omega$ repeated according to their multiplicity, and then correspondingly choose eigenfunctions $\phi_k$ with eigenvalue $\lambda_k$ such that $\{\phi_k\}$ are orthonormal in $L^2(X,\omega^n)$. We set $f_\omega:=\frac{\omega^n/V_\omega}{\omega_X^n/V_{\omega_X}}$,
\begin{theorem}\cite{GPSt,GPSS1,GPSS2,GPSS3} \label{thmgpss}
Given positive numbers $A,K,p$ with $p>1$. For any K\"ahler metric $\omega$ on $X$ with $\int_X\omega\wedge\left(V^{-\frac{1}{n}}_{\omega_X}\omega_X\right)^{n-1}\le A$ and $|f_\omega|_{L^p(X,\frac{\omega_X^n}{V_{\omega_X}})}\le K$, there hold
\begin{itemize}
\item[(0)] for any $x\in X$,
$$-V_\omega\inf_{x,y\in X}G_\omega(x,y)+\int_X|G_\omega(x,\cdot)|\omega^n\le C(n,A,K,p).$$

\item[(1)] for any given $\epsilon\in(0,\frac{n}{n-1})$ and any $x\in X$, 
$$V_\omega^{\frac{n}{n-1}-\epsilon}\int_X|G_\omega(x,\cdot)|^{\frac{n}{n-1}-\epsilon}\frac{\omega^n}{V_\omega}\le C(n,A,K,p,\epsilon).$$

\item[(2)] for any given $\epsilon\in(0,\frac{2n}{2n-1})$ and any $x\in X$,
$$V_\omega^{\frac{2n}{2n-1}-\epsilon}\int_X|\nabla_\omega G_\omega(x,\cdot)|^{\frac{2n}{2n-1}-\epsilon}\frac{\omega^n}{V_\omega}\le C(n,A,K,p,\epsilon).$$

\item[(3)] for any given $\epsilon\in(0,\infty)$ and any $x\in X$, $R\in(0,1]$,
$$\frac{Vol_\omega(B_\omega(x,R))}{V_\omega}\ge c(n,A,K,p,\epsilon)R^{2n+\epsilon}.$$

\item[(4)] for any given $\epsilon\in(0,\frac{n}{n-1})$ and any $u\in W^{1,2}(X)$ (set $\bar u:=\int_Xu\frac{\omega^n}{V_\omega}$),
$$\left(\int_X|u-\bar u|^{2(\frac{n}{n-1}-\epsilon)}\frac{\omega^n}{V_\omega}\right)^{\frac{1}{\frac{n}{n-1}-\epsilon}}\le C(n,A,K,p,\epsilon)\int_X|\nabla u|^2\frac{\omega^n}{V_\omega}.$$

\item[(5)] for any given $\epsilon\in(0,\infty)$ and any $t>0$, $x\in X$,
$$H_\omega(x,x,t)\le\frac{1}{V_\omega}+\frac{C(n,A,K,p,\epsilon)}{V_\omega t^{n+\epsilon}}.$$

\item[(6)] for any given $\epsilon\in(0,\frac1n)$ and any $k\ge1$, 
$$\lambda_k\ge c(n,A,K,p,\epsilon) k^{\frac1n-\epsilon}.$$

\item[(7)] for any given $\epsilon\in(0,+\infty)$ and any $k\ge1$, 
$$|\phi_k|_{L^\infty(X)}^2\le \frac{C(n,A,K,p,\epsilon)}{V_\omega} \lambda_k^{n+\epsilon}.$$
\end{itemize}
\end{theorem}
We remark that Guo-Phong-Song-Sturm actually proved these results with corresponding exponents under the more general $L^1(\log L)^{q>n}$ volume density condition, and items (1)-(4) can also be found in Guedj-T\^o \cite{GT}, Vu \cite{V} and Nguyen-Vu \cite{NV}. It's clear that item (3) implies a diameter bound for $(X,\omega)$ and then precompactness of K\"ahler manifolds satisfying the condition in Theorem \ref{thmgpss}. Moreover, it's also clear from the arguments in the above-mentioned works that the dependence on the upper bound of $\int_X\omega\wedge(V^{-\frac{1}{n}}_{\omega_X}\omega_X)^{n-1}$ can be replaced by dependence on two positive numbers $\alpha$ and $A_\alpha$ satisfying $\int_Xe^{-\alpha\psi}\frac{\omega_X^n}{V_{\omega_X}}\le A_\alpha$ for all $\psi\in PSH(X,\chi)$ with $\sup_X\psi=0$, where $\chi$ is a smooth (not necessarily positive) representative of $[
\omega]$.

The primary motivation of this paper is to discuss the possibility of extending the involved exponents in conclusions of Theorem \ref{thmgpss} to optimal ones, which, if hold, should have important geometric consequences. In particular, the volume noncollapsing estimate in the item (3) with $\epsilon=0$ shall be essential in understanding the metric structure of limit spaces of comapct K\"ahler  manifolds under conditions in Theorem \ref{thmgpss}, while the (higher) eigenvalue estimate in the item (6) with $\epsilon=0$ will match the fundamental Weyl's asymptotic formula on eigenvalues in the order of $k$. Actually, if we allow the uniform constants to additionally depend on a lower bound of Ricci curvature, then the items (3-6) with $\epsilon=0$ therein in Theorem \ref{thmgpss} keep true uniformly and they are optimal (concerning those exponents), see \cite{P.Li} for these classical results.  In the present setting of no Ricci curvature lower bound being required, to obtain improved/optimal geometric estimates, a natural idea is to seek improved/optimal uniform integral estimates for Green's function, since, according to Guo-Phong-Song-Sturm's works, all the items (2)-(7) are relied on item (1) in an essential way. However, simply using the asymptotic of $G_\omega(x,y)$ near $x=y$, we know $|G_\omega(x,\cdot)|^{\frac{n}{n-1}}$ is no longer integrable on $(X,\omega)$, and, for an increasing continuous function $h:\mathbb R_+\to\mathbb R_+$, if we set $\bar h(v):=vh^{\frac1n}(v)$, then the function
$$\frac{|G_\omega(x,\cdot)|^{\frac{n}{n-1}}}{\overline h(\log(|G_\omega(x,\cdot)|+3))}$$ 
is integrable on $(X,\omega)$ if and only if $\int_1^\infty\frac{dv}{\overline h(v)}<+\infty$. Therefore, simultaneously considering the scaling invariance, the possibly optimal uniform integral bound for Green's function should be a uniform bound on 
$$V_{\omega}^{\frac{n}{n-1}}\int_X\frac{|G_{\omega}(x,\cdot)|^{\frac{n}{n-1}}}{\overline h(\log(V_{\omega} |G_{\omega}(x,\cdot)|+C))}\frac{\omega^n}{V_\omega},$$
where $C$ is a positive number, $h:\mathbb R_+\to\mathbb R_+$ is an increasing continuous function and $\bar h(v):=vh^{\frac1n}(v)$, satisfying $\int_1^\infty\frac{dv}{\overline{h}(v)}<+\infty$;  further, given Theorem \ref{thmgpss} (0), bounding this integral is equivalent to bounding
$$V_{\omega}^{\frac{n}{n-1}}\int_X\frac{\mathcal G_{\omega}(x,\cdot)^{\frac{n}{n-1}}}{\overline h(\log(V_{\omega} \mathcal G_{\omega}(x,\cdot)+1))}\frac{\omega^n}{V_\omega},$$
where 
\begin{equation}\label{mathcalg}
\mathcal G_{\omega}(x,\cdot):=G_{\omega}(x,\cdot)-\inf_{x,y\in X}G_{\omega}(x,y)+2V_{\omega}^{-1}
\end{equation}
is a positive Green's function. For convenience, we will work with $\mathcal G_{\omega}$.

\subsection{Main results}
We now state our main result on integral bound for Green's function, which is a nearly optimal version according to the above point of view.

\begin{theorem}\label{thmgreen}
Given positive numbers $\alpha,A,K,p$ with $p>1$ and a Lipschitz increasing function $h:\mathbb R_{+}\to\mathbb R_{+}$ with 
\begin{equation}\label{h}\int_{1}^\infty\frac{dv}{\overline h(v)}<+\infty, \,\,\,where \,\,\,\bar h(v):=vh^{\frac1n}(v),
\end{equation} 
there is a positive number $C(n,\alpha,A,K,p,h)$ satisfying the following. For any K\"ahler metric $\omega$ on $X$ with $|f_\omega|_{L^p(X,\frac{\omega_X^n}{V_{\omega_X}})}\le K$ and, for some smooth real $(1,1)$-form $\chi\in[\omega]$, $\int_Xe^{-\alpha\psi}\frac{\omega_X^n}{V_{\omega_X}}\le A$ for all $\psi\in PSH(X,\chi)$ with $\sup_X\psi=0$, there hold, for any $x\in X$,
\begin{itemize}
\item[(i)]
$$V_{\omega}^{\frac{n}{n-1}}\int_X\frac{\mathcal G_{\omega}(x,\cdot)^{\frac{n}{n-1}}}{\overline h^{\frac{n}{n-1}}(\log(V_{\omega} \mathcal G_{\omega}(x,\cdot)+1))}\frac{\omega^n}{V_\omega}\le C(n,\alpha,A,K,p,h);$$
\item[(ii)]
$$V_{\omega}^{\frac{2n}{2n-1}}\int_X\frac{|\nabla \mathcal G_{\omega}(x,\cdot)|^{\frac{2n}{2n-1}}}{\overline h^{\frac{2n}{2n-1}}(\log(V_{\omega} \mathcal G_{\omega}(x,\cdot)+1))}\frac{\omega^n}{V_\omega}\le C(n,\alpha,A,K,p,h),$$
here and hencefore, we take $\nabla$ and $|\cdot|$ with respect to $\omega$.
\end{itemize}

\end{theorem}

In our proof of Theorem \ref{thmgreen}, the fundamental role is played by the method of auxiliary complex Monge-Amp\`ere equation recently developed by Guo-Phong-Song-Sturm \cite{GPSS1,GPSt} (also see \cite{GT}). Comparing with the argument of Guo-Phong-Song-Sturm \cite{GPSS1,GPSt}, however, in our case we have to use a modified auxiliary complex Monge-Amp\`ere equation to handle the issue of the involved power of Green's function exactly being $\frac{n}{n-1}$ (in previous works \cite{GT,GPSS1,GPSt}, this power is strictly less than $\frac{n}{n-1}$), which may be regarded as in the critical state since $\mathcal G(x,\cdot)^{\frac{n}{n-1}}$ is no longer integrable. By using this modified auxiliary complex Monge-Amp\`ere equation, our argument will be somehow more direct without using another auxiliary Laplacian equation nor iteration argument; moreover, since the involved function $h$ in the above theorem is an abstract one, which of course will certainly appear in the auxiliary complex Monge-Amp\`ere equation, we need to bound the entropy of auxiliary K\"ahler metrics with respect to a suitable abstract norm function and subtly construct a suitable abstract increasing function to apply Young's inequality, see Section \ref{sect-green-pf} for details. Our method can be applied to prove an analogous result under the more general $L^1(\log L)^q$ volume density condition, see Theorem \ref{thmgreen-q}.

As applications of Theorem \ref{thmgreen}, we shall prove the following geometric estimates, extending items (3-7) in Theorem \ref{thmgpss}.

\begin{theorem}\label{thmgeom}
Given positive numbers $\alpha,A,K,p,\epsilon$ with $p>1$, there are positive numbers $C=C(n,\alpha,A,K,p,\epsilon)\ge1$ and $c=c(n,\alpha,A,K,p,\epsilon)\le1$ satisfying the following. For any K\"ahler metric $\omega$ on $X$ with $|f_\omega|_{L^p(X,\frac{\omega_X^n}{V_{\omega_X}})}\le K$ and, for some smooth real $(1,1)$-form $\chi\in[\omega]$, $\int_Xe^{-\alpha\psi}\frac{\omega_X^n}{V_{\omega_X}}\le A$ for all $\psi\in PSH(X,\chi)$ with $\sup_X\psi=0$, there hold
\begin{itemize}

\item[(a)] (Local volume non-collapsing) For any $x\in X$ and $R\in(0,1]$,
$$\frac{Vol_\omega(B_\omega(x,R))}{V_\omega}\ge \frac{c\,R^{2n}}{(1-\log R)^{2n+\epsilon}}.$$

\item[(b)] (Sobolev-type inequality) For any $u\in W^{1,2}(X)$ (set $\bar u:=\int_Xu\frac{\omega^n}{V_\omega}$),
$$\int_XN\left(\frac{c\,|u-\bar u|^{2}}{\int_X|\nabla u|^2\frac{\omega^n}{V_\omega}}\right)\frac{\omega^n}{V_\omega}\le C,$$
where $N:\mathbb R_+\to\mathbb R_+,\,\,N(v)=\frac{v^{\frac{n}{n-1}}}{\log^{\frac{2n}{n-1}+\epsilon}(v+3)}$.

\item[(c)] (On diagonal estimate for heat kernel) For any $x\in X$ and $t>0$, 
$$H_\omega(x,x,t)\le\frac{1}{V_\omega}+\frac{C}{V_\omega}\frac{\log^{2n+\epsilon}(3+\frac1t)}{t^{n}}.$$

\item[(d)] (Growth estimate for higher eigenvalue) For any $k\ge1$, 
$$\lambda_k\ge \frac{c \,k^{\frac1n}}{\log^{2+\epsilon}(k+3)}.$$

\item[(e)] (Eigenfunction estimate) For any $k\ge1$, 
$$|\phi_k|_{L^\infty(X)}^2\le \frac{C}{V_\omega} \lambda_k^{n}\log^{2n+\epsilon}(\lambda_k+3).$$
\end{itemize}
\end{theorem}

The above item (a) is proved using Theorem \ref{thmgreen} (ii); we can actually prove a slightly stronger version of local volume non-collapsing estimate, see Proposition \ref{prop-vol-explicit}. We mention that one may recover the Sobolev inequality in Theorem \ref{thmgpss} (4) from the above item (b), since $N(v)\ge c(n,\epsilon)v^{\frac{n}{n-1}-\epsilon}$ for any $v\ge1$. Actually, we will prove a general family of Sobolev-type inequalities containing item (b) as a particular piece, see Theorem \ref{Sobolev1}.

\begin{rem}\label{rem-scale}
One can easily have a scale-invariant version of the above geometric estimates. For example, if we set, for a K\"ahler class $[\omega]$, 
$$\alpha_{[\omega]}:=\sup\left\{\beta>0|\exists\,\,\mathrm{a}\,\,\mathrm{smooth}\, \chi\in[\omega] \,\,\mathrm{s.t.} \sup\left\{\int_Xe^{-\beta\psi}\frac{\omega_X^n}{V_{\omega_X}}|\psi\in\mathrm{PSH}(X,\chi), \sup_X\psi=0\right\}<\infty\right\},$$
then by a simple rescaling method (precisely, by considering the rescaled metric $\alpha_{[\omega]}\omega$), for positive numbers $A,K,p,\epsilon$ with $p>1$ there are positive numbers $C=C(n,A,K,p,\epsilon)$ and $c=c(n,A,K,p,\epsilon)$ satisfying the following: for any K\"ahler metric $\omega$ on $X$ with $|f_\omega|_{L^p(X,\frac{\omega_X^n}{V_{\omega_X}})}\le K$ and, for some smooth real $(1,1)$-form $\chi\in[\omega]$, $\int_Xe^{-\frac{\alpha_{[\omega]}}{2}\psi}\frac{\omega_X^n}{V_{\omega_X}}\le A$ for all $\psi\in PSH(X,\chi)$ with $\sup_X\psi=0$, there hold
\begin{itemize}

\item[(b')] (Sobolev-type inequality) For any $u\in W^{1,2}(X)$ (set $\bar u:=\int_Xu\frac{\omega^n}{V_\omega}$),
$$\int_XN\left(\frac{c\,|u-\bar u|^{2}}{\alpha_{[\omega]}^{-1}\int_X|\nabla u|^2\frac{\omega^n}{V_\omega}}\right)\frac{\omega^n}{V_\omega}\le C,$$
where $N:\mathbb R_+\to\mathbb R_+,\,\,N(v)=\frac{v^{\frac{n}{n-1}}}{\log^{\frac{2n}{n-1}+\epsilon}(v+3)}$.

\item[(d')] (Growth estimate for higher eigenvalue) For any $k\ge1$, 
$$\lambda_k\ge \frac{c\,\alpha_{[\omega]} \,k^{\frac1n}}{\log^{2+\epsilon}(k+3)}.$$

\item[(e')] (Eigenfunction estimate) For any $k\ge1$, 
$$|\phi_k|_{L^\infty(X)}^2\le \frac{C}{V_\omega} (\alpha_{[\omega]}^{-1}\lambda_k)^{n}\log^{2n+\epsilon}(\alpha_{[\omega]}^{-1}\lambda_k+3).$$
\end{itemize}
\end{rem}

\begin{rem}
As in \cite{GPSS1,GPSS2}, geometric estimates in Theorems \ref{thmgreen} and \ref{thmgeom}  stated under $L^p$ volume density condition have their analog under $L^1(\log L)^{n+\epsilon}$ volume density condition, see Sections \ref{sect-green-pf}, \ref{sect-nabla} and \ref{sect-volume}.
\end{rem}

\subsection{Applications}
We now introduce two important geometric settings where our main results can be applied.
\subsubsection{K\"ahler-Ricci flow}
In the study of the Ricci flow, one of the most fundamental tools is the local volume noncollapsing, which is the key to applying convergence theory of Riemannian geometry. Specializing to the compact K\"ahler case, recent years have seen a series of significant progresses on the local volume noncollasping of the long-time solutions to the K\"ahler-Ricci flow, solving fundamental problems of diameter bound and Gromov-Hausdorff compactness.
As an immediate application of our main results, we shall provide an improved local volume noncollapsing property for the long-time K\"ahler-Ricci flow. Let $(X,\omega_X)$ be an $n$-dimensional minimal K\"ahler manifold (recall that a minimal K\"ahler manifold means a compact K\"ahler manifold of nef canonical line bundle $K_X$, i.e. $c_1(K_X)$ is in the closure of the K\"ahler cone of $X$), and consider the K\"ahler-Ricci flow starting from an initial K\"ahler metric $\omega_0$ on $X$,
\begin{equation}\label{KRF1}
    \begin{cases}
       \partial_t\omega(t)=
        -\mathrm{Ric}(\omega(t))-\omega(t),\\
        \omega(0)=\omega_0,
    \end{cases}
\end{equation}
which exists for $t\in[0,\infty)$ by the maximal existence theorem \cite{C,TZ,Ts}. It is a fundamental problem to understand geometric properties of the long-time K\"ahler-Ricci flow on a minimal K\"ahler manifold, which has a natural deep  connection to Abundance Conjecture in algebraic geometry \cite{ST1,ST2,ST3}. Recall that Abundance Conjecture predicts that a minimal projective/K\"ahler manifold $X$ should have a semi-ample canonical line bundle $K_X$. It is a fundamental fact that a semi-ample $K_X$ must be nef. In their pioneered works, Song-Tian \cite{ST1,ST2,ST3} studied the K\"ahler-Ricci flow on a compact K\"ahler manifold of semi-ample canonical line bundle and proved the weak convergence to a generalized K\"ahler-Einstein metric on the canonical model. Song-Tian further conjectured that the convergence should happen in the global sense of Gromov-Hausdorff, and there are important progresses, see \cite{FZz,GSW,HLT,JS,LTo,STZ,TZzl,TZzl1,TWY,W,Zys,ZZ} and references therein. 

Among a variety of ideas and techniques developed in the field, let's focus on local volume noncollapsing of the flow, which plays a central role in recent important progresses. To begin with, we state the one obtained by Jian-Song \cite[Corollary 1.1]{JS} in the semi-ample canonical line bundle case: let $\omega(t)$ be a long-time solution to the K\"ahler-Ricci flow \eqref{KRF1} on a compact K\"ahler manifold of semi-ample canonical line bundle, then there is a positive number $c$ such that for any $t\in[0,\infty)$, $x\in X$ and $R\in(0,diam(X,\omega(t))]$,
\begin{equation}\label{vol-2n}
\frac{Vol_{\omega(t)}(B_{\omega(t)}(x,R))}{Vol_{\omega(t)}(X)}\ge cR^{2n}.
\end{equation}

In general, it is conjectured that the long-time K\"ahler-Ricci flow on a minimal K\"ahler manifold should Gromov-Hausdorff converge to a unique compact metric space induced by a singular metric of K\"ahler-Einstein type \cite{ST3}. Very recently, by developing a striking volume noncollapsing estimate, it has been proved that the long-time K\"ahler-Ricci flow on a minimal K\"ahler manifold subsequently  converges to a compact metric space in Gromov-Hausdorff topology \cite{GPSS1,GPSS2,GT,V}. Precisely, for a long-time solution $\omega(t)$ to the K\"ahler-Ricci flow \eqref{KRF1} on a minimal K\"ahler manifold and any $\epsilon>0$, there is a positive number $c_\epsilon$ such that for any $t\in [0,\infty)$, $x\in X$ and $R\in(0,1]$,
\begin{equation}\label{vol-2n+}
\frac{Vol_{\omega(t)}(B_{\omega(t)}(x,R))}{Vol_{\omega(t)}(X)}\ge c_\epsilon R^{2n+\epsilon}.
\end{equation}
This is first proved by Guo-Phong-Song-Sturm \cite{GPSS1} for minimal K\"ahler manifolds of nonnegative Kodaira dimension, and then extended to general minimal K\"ahler manifolds by Vu \cite{V}, Guedj-T\^o \cite{GT} and Guo-Phong-Song-Sturm \cite{GPSS2}. 

Motivated by the Abundance Conjecture, it is natural to expect that for a long-time solution $\omega(t)$ to the K\"ahler-Ricci flow \eqref{KRF1} on a minimal K\"ahler manifold, the volume noncollapsing estimate \eqref{vol-2n+} can be further improved to \eqref{vol-2n}, which, if holds, provides positive evidence for Abundance Conjecture and may have important applications in studying limit spaces of the flow. This actually serves as another motivation for this work. As an application of the main result, we will show the following version of volume noncollapsing estimate, which improves \eqref{vol-2n+} and could be seen as a partial result for this expectation.

\begin{theorem}\label{thm-infinite}
Let $(X,\omega_X)$ be an $n$-dimensional minimal K\"ahler manifold and $\omega(t)$, $t\in[0,\infty)$, the long-time solution to the K\"ahler-Ricci flow \eqref{KRF1} starting from an initial K\"ahler metric $\omega_0$. Then for any $\epsilon>0$, there exists a positive number $\tilde c=\tilde c(X,\omega_0,\epsilon)$ such that for any $t\in[0,\infty)$, $x\in X$ and $R\in(0,1]$, there holds
\begin{equation}\label{vol-est-1}
\frac{Vol_{\omega(t)}(B_{\omega(t)}(x,R))}{Vol_{\omega(t)}(X)}\ge \frac{\tilde cR^{2n}}{(-\log R+1)^{2n+\epsilon}}.
\end{equation}
\end{theorem}
Similar result holds for a finite-time solution to the (unnormalized) K\"ahler-Ricci flow, assuming the (global) volume is noncollapsing.

\begin{theorem}\label{thm-finite}
Let $(X,\omega_X)$ be an $n$-dimensional compact K\"ahler manifold and $\omega(t)$, $t\in[0,T)$, be a finite-time solution to the (unnormalized) K\"ahler-Ricci flow starting from an initial K\"ahler metric $\omega_0$, 
\begin{equation}\label{KRF2}
    \begin{cases}
        \partial_t\omega(t)=
        -\mathrm{Ric}(\omega(t)),\\
        \omega(0)=\omega_0.
    \end{cases}
\end{equation}
Assume $\lim _{t\to T}Vol_{\omega(t)}(X)>0$. Then for $\epsilon>0$, there exists a positive number $\tilde c=\tilde c(X,\omega_0,\epsilon)$ such that for any $t\in[0,\infty)$, $x\in X$ and $R\in(0,1]$, there holds
\begin{equation}\label{vol-est-finite-1}
Vol_{\omega(t)}(B_{\omega(t)}(x,R))\ge \frac{\tilde cR^{2n}}{(-\log R+1)^{2n+\epsilon}}.
\end{equation}
\end{theorem}

\subsubsection{K\"ahler family}\label{subsect-family} In this subsection, we consider the setting of K\"ahler family studied in \cite{DGG,GT}. Let $\mathcal X$ be an irreducible and reduced complex space and $\pi:\mathcal X\to2\mathbb D$ a proper holomorphic surjection with connected fibers such that for each $s\in2\mathbb D^*$, the fiber $X_s:=\pi^{-1}(s)$ over $s$ is an $n$-dimensional compact K\"ahler manifold, and $X_0:=\pi^{-1}(0)$ is an irreducible and reduced complex space. A smooth form $\beta$ on $\mathcal X$ is called a relative K\"ahler form on $\mathcal X$ if the restriction $\beta_s:=\beta|_{X_s}$ is a K\"ahler form on $X_s$ for each $s\in2\mathbb D^*$. In the followings we always assume that $V_{\beta_s}:=Vol_{\beta_s}(X_s)$ is uniformly bounded away from $0$ and $+\infty$ for $s\in2\mathbb D^*$. Applying the main result to the K\"ahler family, we have the following extension of \cite[Theorem]{GT}.

\begin{theorem}\label{thmgreen-family}
Let $\mathcal X\to 2\mathbb D$ be a K\"ahler family and $\beta$ a relative K\"ahler form on $\mathcal X$ as defined above. For positive numbers $A,B,p$ with $p>1$ and a Lipschitz increasing function $h:\mathbb R_{+}\to\mathbb R_{+}$ with
\begin{equation}\int_{1}^\infty\frac{dv}{\overline h(v)}<+\infty, \,\,\,where \,\,\,\bar h(v):=vh^{\frac1n}(v),\nonumber
\end{equation}  
there is a positive number $C=C(n,p,A,B,h)$ satisfying the following. If $\omega$ is a relative K\"ahler form on $\mathcal X$ such that for each $s\in2\mathbb D^*$, 
$$[\omega_s]\le A[\beta_s]\,\,\, in\,\,\, H^{1,1}(X_s,\mathbb R)\,\,\, and \,\,\,\int_{X_s}f_s^p\frac{\beta_s^n}{V_{\beta_s}}\le B,$$
then for each $s\in\mathbb D^*$, $x\in X_s$, there hold
\begin{itemize}
\item[(i')]
$$V_{\omega_s}^{\frac{n}{n-1}}\int_{X_s}\frac{\mathcal G_{\omega_s}(x,\cdot)^{\frac{n}{n-1}}}{\bar h^{\frac{n}{n-1}}(\log(V_{\omega_s}\mathcal G_{\omega_s}(x,\cdot)+1))}\frac{\omega_s^n}{V_{\omega_s}}\le C(n,p,A,B,h);$$

\item[(ii')]
$$V_{\omega_s}^{\frac{2n}{2n-1}}\int_{X_s}\frac{|\nabla_{\omega_s}\mathcal G_{\omega_s}(x,\cdot)|^{\frac{2n}{2n-1}}}{\bar h^{\frac{2n}{2n-1}}(\log(V_{\omega_s}\mathcal G_{\omega_s}(x,\cdot)+1))}\frac{\omega_s^n}{V_{\omega_s}}\le C(n,p,A,B,h).$$
\end{itemize}
\end{theorem}
Given these Green's function integral bounds, the uniform geometric estimates listed in Theorem \ref{thmgeom} can also be proved for $\omega_s,$ $s\in\mathbb D^*$. We omit the precise statement here.

\subsection{Organization of the remaining part}
\begin{itemize}
\item In Section \ref{sect-pre}, we will collect necessary preliminaries on uniform estimate on K\"ahler potentials.
\item In Sections  \ref{sect-green-pf} and \ref{sect-nabla}, we prove uniform integral estimates for Green's function and its gradient under $L^p$ conditions stated in Theorem \ref{thmgreen} (i). An analog under  $L^1(\log L)^q$ condition is also proved, see Theorem \ref{thmgreen-q}.

\item In sections \ref{sect-volume}-\ref{sect-eigen}, we prove geometric estimates stated in items (a)-(e) of Theorem \ref{thmgeom}, respectively.

\item In Section \ref{sect-vol-pf} we prove Theorems \ref{thm-infinite} and \ref{thm-finite} on the K\"ahler-Ricci flow. We shall also provide a generalization to a twisted K\"ahler-Ricci flow on non-minimal K\"ahler manifolds, see Theorem \ref{thm-vol-tkrf}.
\item In Section \ref{sect-family}, we present the applications to a general K\"ahler family.
\item In Appendix \ref{app-mto}, we present a Moser-Trudinger-Onofri type inequality in the critical case.
\item In Appendix \ref{app-Linfty}, we present a proof of Theorem \ref{Linfty}.
\end{itemize}

\section{Preliminaries: uniform $L^\infty$ estimate for K\"ahler potentials}\label{sect-pre}
Throughout this paper, a fundation is the uniform $L^\infty$ estimate for K\"ahler potentials, which we collect in this section.

The complex Monge-Amp\`ere equation plays a central role in the study of K\"ahler geometry, in which a fundamental part is the uniform $L^\infty$ estimate on the potential functions. In his celebrated solution to Calabi conjecture, Yau \cite{Y} established the $L^\infty$ estimate for the complex Monge-Amp\`ere equation on an $n$-dimensional compact K\"ahler manifold using Moser iteration, which actually works for the case of right hand side being in $L^p,\,p>n$. Using pluripotential theory, Kolodziej \cite{K} obtained $L^\infty$ estimate for the complex Monge-Amp\`ere equation on an $n$-dimensional compact K\"ahler manifold in terms of certain Orlicz norm of right hand side, containing the case of right hand side being in $L^p,\,p>1$. More developments can be found in \cite{BEGZ,DP,DGG,DZ,EGZ1,EGZ2,FGS,GL,GPT,GPTW,Liuj} etc.. In our case, to get geometric estimates uniform for family of K\"ahler manifolds, the $L^\infty$ estimate has to be of explicit dependence on the background quantities. Precisely, we will need the following version of $L^\infty$ estimate, where (and hencefore) we will use the following notations: let $\omega$ be a K\"ahler metric on an $n$-dimensioanl compact K\"ahler manifold $(X,\omega_X)$ and $\Phi:\mathbb R_+\to\mathbb R_+$ be a continuous function, then 
\begin{itemize}
\item set 
$$f_\omega:=\frac{\omega^n/V_\omega}{\omega_X^n/V_{\omega_X}}\,\,\,\,\,\, and \,\,\,\,\,\,N_\Phi(\omega):=\int_X\Phi(\log( f_\omega+1))f_\omega\frac{\omega_X^n}{V_{\omega_X}};$$
\item call $\Phi$ a \emph{K-function} if $\Phi:\mathbb R_+\to\mathbb R_+$ is a Lipschitz increasing function satisfying $\lim_{v\to+\infty}\Phi(v)=+\infty$ and $\int_{1}^{\infty}\Phi(v)^{-\frac1n}dv=:C_{\Phi}<+\infty$.
\end{itemize} 

\begin{theorem}\label{Linfty}
Let $(X,\omega_X)$ be an $n$-dimensioanl compact K\"ahler manifold, and $\theta$ a closed real $(1,1)$-form representating a K\"ahler class. Let $\omega=\theta+\sqrt{-1}\partial\bar\partial\varphi$ be a K\"ahler metric with $\sup_X\varphi=0$. Assume that
\begin{itemize}
\item[(H1)] there exist two positive numbers $\alpha$ and $A_\alpha$ such that for each $\psi\in PSH(X,\theta)$ there holds
$$\int_{X}e^{\alpha(\sup_X\psi-\psi)}\frac{\omega_X^n}{V_{\omega_X}}\le A_\alpha;$$
\item[(H2)] there exist a K-function $\Phi:\mathbb R_+\to\mathbb R_+$ and a positive number $B_\Phi$ such that
$$N_\Phi(\omega)\le B_\Phi,$$
\end{itemize}
then 
$$|u_\theta-\varphi|_{L^\infty}\le C(n,\alpha,A_\alpha,\Phi,B_\Phi,C_\Phi),$$
where $u_\theta:=\sup\{u\in PSH(X,\theta)|u\le 0\}$ is the envelope of $(X,\theta)$.
\end{theorem}
This is (essentially) proved in Liu \cite[Section7]{Liuj} basing on the PDE approach developed in a recent breakthrough of Guo-Phong-Tong \cite{GPT} (also see \cite{GPTW}). The $L^\infty$ estimate under the general condition of volume density stated in (H2) goes back to Kolodziej's fundamental work \cite{K}.
Since Theorem \ref{Linfty} is fundamental in this work and its setup and dependence of the required uniform bound are slightly different from \cite{K,GPT,GPTW,Liuj}, for convenience and completeness, in Appendix \ref{app-Linfty} we will present a proof of Theorem \ref{Linfty} by using Guo-Phong-Tong's PDE approach. Actually,  the proof will follow \cite{GPT,GPTW,Liuj} closely, and we just need to make the dependence relation slightly sharper (precisely, get rid of the dependence on the background metric $\omega_X$ and the reference form $\theta$) and chase the involved constants carefully. In particular, as in \cite[Theorem 1.1]{DGG}, the uniform bound $C(n,\alpha,A_\alpha,\Phi,B_\Phi,C_\Phi)$ will be expressed in terms of its parameters rather explicitly, see \eqref{Lambda1}, \eqref{balpha}, \eqref{C0} and \eqref{C}.

\section{Integral estimate on Green's function: proof of Theorem \ref{thmgreen} (i)}\label{sect-green-pf}
This section contains a proof of our new integral estimate on Green's function stated in Theorem \ref{thmgreen} (i). To roughly see the difference between our proof and previous ones, let's summarize the main steps in the proof of Theorem \ref{thmgpss} (1), given in \cite{GPSt,GPSS1,GT,NV}: the first step is to prove a uniform $L^k$ bound, $k<1+\frac1n$, for the Green's function by using an auxiliary complex Monge-Amp\`ere equation and an auxiliary Laplace equation, in which the right hand side of the involved auxiliary complex Monge-Amp\`ere equation has a uniform $L^{p'}$ bound for some $p'<p$; the second step is iterating the first step to obtain the uniform $L^k$ bound for the Green's function, for any $k<\frac{n}{n-1}$. In our proof of Theorem \ref{thmgreen} (i) under the $L^p$ condition, since the Green's function $\mathcal G_\omega$ has the critical power $\frac{n}{n-1}$, we have to work with a modified auxiliary complex Monge-Amp\`ere equation, whose right hand side, however, is never of uniform $L^{1+\epsilon}$ bound but of uniform entropy bound with respect to certain suitably-constructed K-function (this is the reason that the general version of $L^\infty$ estimate in Theorem \ref{Linfty} is necessary for us). Moreover, our argument does not involve the auxiliary Laplace equation nor the iteration method, and hence is somehow more direct. Applying similar idea to the case under $L^1(\log L)^q$ volume density condition provides the following 

\begin{theorem}\label{thmgreen-q}
Given positive numbers $\alpha,A,K,q$ with $q>n$ and a Lipschitz increasing function $h:\mathbb R_{+}\to\mathbb R_{+}$ with $\int_{1}^\infty\frac{dv}{vh^{\frac1n}(v)}<+\infty$ and $\frac{v^{q-n}}{h(v)}$ being increasing for $v\ge1$.
There is a positive number $C(n,\alpha,A,K,q,h)$ satisfying the following. For any K\"ahler metric $\omega$ on $X$ with $|f_\omega|_{L^1(\log L)^q(X,\omega_X)}\le K$ and, for some smooth real $(1,1)$-form $\chi\in[\omega]$, $\int_Xe^{-\alpha\psi}\omega_X^n\le A$ for all $\psi\in PSH(X,\chi)$ with $\sup_X\psi=0$, there hold, any $x\in X$,
\begin{itemize}
\item[(i)]
$$V_{\omega}^{\frac{nq}{nq-q+n}}\int_X\frac{\mathcal G_{\omega}(x,\cdot)^{\frac{nq}{nq-q+n}}}{h^{\frac{q}{nq-q+n}}(\log(V_{\omega} \mathcal G_{\omega}(x,\cdot)+1))}\frac{\omega^n}{V_\omega}\le C(n,\alpha,A,K,q,h);$$
\item[(ii)]
$$V_{\omega}^{\frac{2nq}{2nq-q+n}}\int_X\frac{|\nabla \mathcal G_{\omega}(x,\cdot)|^{\frac{2nq}{2nq-q+n}}}{h^{\frac{2q}{2nq-q+n}}(\log(V_{\omega} \mathcal G_{\omega}(x,\cdot)+1))}\frac{\omega^n}{V_\omega}\le C(n,\alpha,A,K,q,h).$$
\end{itemize}
\end{theorem}

We will handle Theorem \ref{thmgreen} (i) and Theorem \ref{thmgreen-q} (i) simultaneously. To this end, we consider two cases and define necessary notations as follows:
\begin{itemize}
\item[\textbf{Case I}]: Assume setup of Theorem \ref{thmgreen} with the function $h$. Set $k_{\textbf{I}}:=\frac{n}{n-1}$. For $a\ge 1$, define two Lipschitz functions $j_{\textbf{I}}=j_{\textbf{I},h,a}(v):=\bar h(\log(v+a))$ and
$$J_{\textbf{I}}(v)=J_{\textbf{I},h,a}(v):=\frac{v^{k_{\textbf{I}}-1}}{j_{\textbf{I}}^{k_{\textbf{I}}}(v)}=\frac{v^{\frac{1}{n-1}}}{\overline h^{\frac{n}{n-1}}(\log(v+a))},\,\,v\in[1,+\infty).$$
A constant $C_{\textbf{I}}(*)$ stands for a constant depending on $n,\alpha,A,K,p,h$ and $*$.
\item[\textbf{Case II}]: Assume setup of Theorem \ref{thmgreen-q} with the function $h$ satisfying that $\frac{v^{\frac{q-n}{q}}}{h(v)}$ is increasing in $v$. Set $k_{\textbf{II}}:=\frac{nq}{nq-q+n}$. For $a\ge 1$, define two Lipschitz functions $j_{\textbf{II}}(v)=j_{\textbf{II},h,a}(v):=h^{\frac{1}{n}}(v+a)$ and
$$J_{\textbf{II}}(v)=J_{\textbf{II},h,a}(v):=\frac{v^{k_{\textbf{II}}-1}}{j_{\textbf{II}}^{k_{\textbf{II}}}(v)}=\frac{v^{\frac{q-n}{nq-q+n}}}{h^{\frac{q}{nq-q+n}}(v+a)},\,\,v\in[1,+\infty).$$
A constant $C_{\textbf{II}}(*)$ stands for a constant depending on $n,\alpha,A,K,q,h$ and $*$.
\end{itemize}

We shall provide the following unified result, which implies Theorem \ref{thmgreen} (i) and Theorem \ref{thmgreen-q} (i) immediately. For simplicity of notations, we write $\mathcal G=\mathcal G_\omega$.

\begin{prop}\label{prop-green-1}
In each Case $\gamma=\textbf{I}$ and $\textbf{II}$, there is a positive number $C_{\gamma}$ such that
\begin{equation}\label{goal2}
\int_X\mathcal G(x,\cdot)\cdot J_{\gamma}(V_\omega\mathcal G(x,\cdot))\omega^n\le C_{\gamma}.
\end{equation}
\end{prop}

\begin{proof}[Proof of Proposition \ref{prop-green-1}]

Write $\omega=\chi+\sqrt{-1}\partial\bar\partial\varphi$ with $\sup_X\varphi=0$.

For integer $k\gg1$ we fix a smooth positive function $H_k$, which is a smoothing of 
    $\min\{\mathcal{G}(x,\cdot),k\}$. Without loss of generality, we may assume that $H_k$ converges increasingly to 
    $\mathcal{G}(x,\cdot)$ as $k\to\infty$
    and for each $k$, 
    \[
        \hat{H}_k:=V_\omega\cdot H_k\ge1 \,\,\,on\,\,\,X.
    \]
   For  each Case $\gamma=\textbf{I}$ and $\textbf{II}$, define
    \begin{equation}
        A_{\gamma,k}:=\int_X H_k\cdot J_{\gamma}(\hat{H}_k)\omega^n+1=\int_X \hat H_k\cdot J_{\gamma}(\hat{H}_k)\frac{\omega^n}{V_\omega}+1=\int_X \hat H_k\cdot J_{\gamma}(\hat{H}_k)f_\omega\frac{\omega_X^n}{V_{\omega_X}}+1,
    \end{equation}
here recall that $f_\omega:=\frac{\omega^n/V_\omega}{\omega_X^n/V_{\omega_X}}.$ Then our goal is to obtain a unform upper bound for $A_{\gamma,k}$. To this end, we consider the  following complex Monge-Amp\`ere equation
    \[
        \frac{1}{V_\omega}(\chi+\dd\psi_{\gamma,k})^n=\frac{(J_{\gamma}(\hat{H}_k)^n/A_{\gamma,k}+1)}{B_{\gamma,k}}f_{\omega}\frac{\omega_X^n}{V_{\omega_X}}
    \]
    with 
    \[
        \sup_X \psi_k=0, B_{\gamma,k}=\int_X(1+\frac{1}{A_{\gamma,k}}J_{\gamma}(\hat{H}_k)^n)f_\omega\frac{\omega_X^n}{V_{\omega_X}}.
    \]
    Here we may particularly mention that we have involved the quantity $A_{\gamma,k}$ that we want to bound in the right hand side of the above auxiliary equation, which is a \emph{key difference} from auxiliary complex Monge-Amp\`ere equations in previous works \cite{GPSS1,GPSS2,GT,NV}.

    Firstly, observe that $B_{\gamma,k}$ are uniformly bounded for each $\gamma$. Actually, for $\gamma=\textbf I$,
    \begin{equation}
        \begin{aligned}
            J_{\textbf I}(\hat{H}_k)^n&=
            \frac{\hat{H}_k^{\frac{n}{n-1}}}{
                \bar{h}^{\frac{n^2}{n-1}}(\log(\hat{H}_k+a))}
            =\frac{\hat{H}_kJ_{\textbf I}(\hat{H}_k)}{
                \bar{h}^{n}(\log(\hat{H_k}+a))
            }\\
            &\leq \bar{h}^{-n}(\log(1+a))\hat{H}_kJ_{\textbf I}(\hat{H}_k),
        \end{aligned}
    \end{equation}
    then 
    \[
        1\leq B_k \leq 1+\bar{h}^{-n}(\log(1+a))\leq 1+\bar{h}^{-n}(\log2);
    \]
    while for $\gamma=\textbf{II}$,
    \begin{equation}
        \begin{aligned}
            J_{\textbf{II}}(\hat{H}_k)^n&=
            \frac{\hat{H}_k^{\frac{nq-n^2}{nq-q+n}}}{
                h^{\frac{nq}{nq-q+n}}(\hat{H}_k+a)}
            =\frac{\hat{H}_k J_{\textbf{II}}(\hat{H}_k)}{
                \hat{H}_k^{\frac{n^2}{nq-q+n}}h^{
                    \frac{q(n-1)}{nq-q+n}
                }(\hat{H}_k+a)
            }\\
            &\leq h^{-\frac{(n-1)q}{nq-q+n}}(1+a)\hat{H}_kf(\hat{H}_k),
        \end{aligned}
    \end{equation}
    then 
    \[
        1\leq B_k \leq 1+h^{-\frac{(n-1)q}{nq-q+n}}(1+a)\leq 1+h^{-\frac{(n-1)q}{nq-q+n}}(2).
    \]
    
    Denote $\omega_{\gamma,k}:=\chi+\dd\psi_{\gamma,k}$ and  $f_{\gamma,k}:=\frac{1}{B_{\gamma,k}}(J_{\gamma}(\hat{H}_k)^n/A_{\gamma,k}+1)f_\omega$. We are going to uniformly bound the entropy
    $$N_{\Phi_\gamma}(\omega_{\gamma,k})=\int_X\Phi_\gamma(\log(f_{\gamma,k}+1))f_{\gamma,k}\frac{\omega_X^n}{V_{\omega_X}}$$ 
    for suitable K-function $\Phi_\gamma$. In this part we have to discuss case by case.\\
    
\noindent \underline{Uniform bound on $N_{\Phi_{\textbf I}}(\omega_{\textbf I,k})$ in \textbf{Case} \textbf{I}.}
   
    We choose a function $g_{\textbf I}:(0,\infty)\to(0,\infty)$ by 
    $g_{\textbf I}(v):=\min\{h(v),e^v\}$ which is increasing and 
    \[
        \int_1^\infty \frac{1}{vg_{\textbf I}^{1/n}(v)}\mathrm{d}v<\infty.
    \] 
    Define $\bar{g}_{\textbf I}(v):=vg_{\textbf I}^{1/n}(v)$ and $\bar{g}_{\textbf I,\xi}(v)=\bar{g}_{\textbf I}(\xi v)$, where $\xi>0$ is a positive number to be determined.
    Let
    \[
        \Phi_{\textbf I,\xi}(v):=\bar{g}^n_{\textbf I,\xi}(v),
    \]
    which is a K-function for each $\xi>0$.
    We are going to uniformly bound 
    $$N_{\Phi_{\textbf I,\xi}}(\omega_{\textbf{I},k})=\int_X\Phi_{\textbf I,\xi}(\log(f_{\textbf{I},k}+1))f_{\textbf{I},k}\frac{\omega_X^n}{V_{\omega_X}}$$ 
    for some $\xi$.
    Firstly, we note that
        \begin{align}
            &\int_X\Phi_{\textbf I,\xi}(\log(f_{\textbf{I},k}+1))f_{\textbf{I},k}\frac{\omega_X^n}{V_{\omega_X}} \notag\\
            \leq & \int_X(J_{\textbf I}(\hat{H}_k)^n/A_{\textbf{I},k}+1)\bar{g}_{\textbf I,\xi}^n(
                \log(1+(J_{\textbf I}(\hat{H}_k)^n/A_{\textbf{I},k}+1)f_\omega))f_\omega\frac{\omega_X^n}{V_{\omega_X}} \notag\\
            \leq &\int_X (J_{\textbf I}(\hat{H}_k)^n/A_{\textbf{I},k}+1)(\bar{g}_{\textbf I,\xi}^n(2\log(2+J_{\textbf I}(\hat{H}_k)^n/A_{\textbf{I},k}))+
            \bar{g}_\alpha^n(2\log(1+f_\omega)))f_\omega\frac{\omega_X^n}{V_{\omega_X}} \notag\\
            =&\int_X (J_{\textbf I}(\hat{H}_k)^n/A_{\textbf{I},k}+1)(\bar{g}_{\textbf I,\xi}^n(
                2\log(2+J_{\textbf I}(\hat{H}_k)^n/A_{\textbf{I},k})))f_\omega\frac{\omega_X^n}{V_{\omega_X}}\label{1} \\
            &+\int_X (J_{\textbf I}(\hat{H}_k)^n/A_{\textbf{I},k}+1)(
                \bar{g}_{\textbf I,\xi}^n(2\log(1+f_\omega)))f_\omega\frac{\omega_X^n}{V_{\omega_X}}\label{2}
    \end{align}
    where in the second inequality we have applied 
    \begin{equation}\label{ele1}
        \log(1+xy)\leq \log(1+x)+\log(1+y)\,\,\,(x,y>0)
    \end{equation}
    and
    \begin{equation}\label{ele2}
        \bar{g}_{\textbf I,\xi}^n(x+y)\leq \bar{g}^n_{\textbf I,\xi}(2x)+\bar{g}^n_{\textbf I,\xi}(2y).
    \end{equation}
    To estimate \refe{1}, we note
    \[
        \bar{g}_{\textbf I,\xi}^n(2\log(2+J_{\textbf I}(\hat{H}_k)^n/A_{\textbf{I},k}))\leq
        \bar{g}_{\textbf I}^n(\frac{1}{2}\log(\hat{H}_k+a))
    \]
    where we have applied the fact $A_{\textbf{I},k}\geq1$ and for some $\theta>0$,
    \begin{equation}\label{log term}
        \log(2+J_{\textbf{I}}(v)^n)\leq \log(2+v^{\frac{n}{n-1}}/\log^{\frac{n^2}{n-1}}(v+a))
     \leq \theta\log(v+a),\,\,\, for\,\,\,any\,\,\, v\geq1
    \end{equation}
    and then chosen $\xi\leq \frac{1}{4\theta}$.
    Then by $g(v)\leq e^v$ we have 
    \begin{equation}\label{1.1}
        \begin{aligned}
            \int_X \bar{g}_{\textbf I,\xi}^n(
                2\log(&2+J_{\textbf I}(\hat{H}_k)^n/A_{\textbf I,k})))f_\omega\frac{\omega_X^n}{V_{\omega_X}} \leq
            \int_X \bar{g}_{\textbf I}^n(\frac{1}{2}\log(1+\hat{H}_k))f_\omega\frac{\omega_X^n}{V_{\omega_X}}\\
            &=\frac{1}{2^n}\int_X \log^n(1+\hat{H}_k) 
            g_{\textbf I}(\frac{1}{2}\log(1+\hat{H}_k))f_\omega\frac{\omega_X^n}{V_{\omega_X}}\\
            &\leq \frac{1}{2^n}
            \int_X \log^n(1+\hat{H}_k) (1+\hat{H}_k)^{
                \frac{1}{2}
            }f_\omega\frac{\omega_X^n}{V_{\omega_X}}\leq C
        \end{aligned}
    \end{equation}
    where we have applied Theorem \ref{thmgpss} (0), the H\"older inequality and 
    the inequality $\log^n(1+v)\leq Cv^{\frac{1}{2}}$ for any $v\geq1$.
    Note that $g_{\textbf I}(v)\leq h(v)$ implies $\bar{g}_{\textbf I}(v)\leq \bar{h}(v)$, then
    \begin{equation}\label{1.2}
        \begin{aligned}
            &\frac{1}{A_{\textbf I,k}}\int_X
            J_{\textbf I}(\hat{H}_k)^n\bar{g}_{\textbf I,\xi}^n(
                2\log(2+J_{\textbf I}(\hat{H}_k)^n/A_{\textbf I,k}))f_\omega\frac{\omega_X^n}{V_{\omega_X}}\\
            \leq&\frac{1}{A_{\textbf I,k}}\int_X
            J_{\textbf I}(\hat{H}_k)^n\bar{g}_{\textbf I}^n
                (\log(\hat{H}_k+a))f_\omega\frac{\beta^n}{V_\beta}\\
            \leq&\frac{1}{A_{\textbf I,k}}\int_X
            J_{\textbf I}(\hat{H}_k)^n\bar{h}^n
                (\log(\hat{H}_k+a))f_\omega\frac{\omega_X^n}{V_{\omega_X}}\\
                =&\frac{A_{\textbf I,k}-1}{A_{\textbf I,k}}\le1.
        \end{aligned}
    \end{equation}
    Combining \refe{1.1} and \refe{1.2} we have that \refe{1} is uniformly bounded.\\
    To estimate \refe{2}, we observe that 
    \begin{equation}\label{g}
        \begin{aligned}
            \bar{g}^n_{\textbf I,\xi}(2\log(1+f_\omega))&=(2\xi)^n\log^n(1+f_\omega)g_{\textbf I}(2\xi\log(1+f_\omega))\\
            &\leq(2\xi)^n\log^n(1+f_\omega)(1+f_\omega)^{2\xi}\\
            &\leq(2\xi)^n(e^{\xi \log f_\omega}+C(\xi))(1+e^{2\alpha  \log f_\omega})
        \end{aligned}
    \end{equation}
    where we have  chosen $2\xi<1$ and applied the estimates
     $g_{\textbf I}(v)\leq e^v$, $(1+v)^{2\xi}\leq 1+v^{2\xi}$ and  $\log^n(1+v)\leq v^\xi +C(\xi)$ for any $v>0$.
    We define an increasing function $\eta_{\textbf I}:(0,\infty)\to(0,\infty)$ by 
    \[
        \eta_{\textbf I}(z)=\left(\exp\{\frac{1}{2\xi}\bar{g}_{\textbf I}^{-1}(z^{\frac{1}{n}})\}-1\right)^{\frac{p-1}{2}}.
    \]
    It's easy to check 
    \[
        \eta_{\textbf I}^{-1}(w)=\bar{g}_{\textbf I}^n(2\xi\log(w^{\frac{2}{p-1}}+1)).
    \]
    Since 
    \begin{equation}\label{young1}
        \begin{aligned}
            \bar{g}_{\textbf I,\xi}^n(2\log(1+ f_\omega))
            \eta_{\textbf I}(\bar{g}_{\textbf I,\xi}^n(2\log(1+f_\omega)))
            =\bar{g}_{\textbf I,\xi}^n(2\log(1+f_\omega))e^{\frac{p-1}{2}\log f_\omega}\\
            \leq (2\xi)^n(e^{\xi \log f_\omega}+C(\xi))(1+e^{2\xi \log f_\omega})e^{\frac{p-1}{2}\log f_\omega}
        \end{aligned}
    \end{equation}
    and 
    \begin{equation}\label{young2}
        \begin{aligned}
            J_{\textbf I}(\hat{H}_k)^n\eta_{\textbf I}^{-1}(J_{\textbf I}(\hat{H}_k)^n)&=J_{\textbf I}(\hat{H}_k)^n
            \bar{g}_{\textbf I}^n(2\xi\log(J_{\textbf I}^{\frac{2n}{p-1}}(\hat{H}_k)+1))\\
            &\leq J_{\textbf I}(\hat{H}_k)^n\bar{g}_{\textbf I}^n(2\xi \theta\log(\hat{H}_k+a))\\
            &\leq J_{\textbf I}(\hat{H}_k)^n\bar{h}^n(\log(\hat{H}_k+a))\\
            &=\hat{H}_k J_{\textbf I}(\hat{H}_k)
        \end{aligned}
    \end{equation}
    where we have applied the estimate similar to \refe{log term} and chosen
    $\xi\leq \frac{1}{2\theta}$.\\
    By \refe{young1}, \refe{young2} and the generalized Young's inequality with respect to $\eta$, we have
    \begin{equation}\label{young}
        \begin{aligned}
            &J_{\textbf I}(\hat{H}_k)^n\bar{g}_{\textbf I,\xi}^n(2\log(1+f_\omega))\\
            \leq& J_{\textbf I}(\hat{H}_k)^n\eta_{\textbf I}^{-1}(J_{\textbf I}(\hat{H}_k)^n)
            +\bar{g}_{\textbf I,\xi}^n(2\log(1+f_\omega))\eta_{\textbf I}(\bar{g}_{\textbf I,\xi}^n(2\log(1+f_\omega)))\\
            \leq&\hat{H}_k J_{\textbf I}(\hat{H}_k)+(2\xi)^n(e^{\xi \log f_\omega}+C(\xi))(1+e^{2\xi \log f_\omega})e^{\frac{p-1}{2}\log f_\omega}.
        \end{aligned}
    \end{equation}
    Then 
    \begin{equation}
        \begin{aligned}
            \int_X& (J_{\textbf I}(\hat{H}_k)^n/A_{\textbf I,k}+1)(
                \bar{g}_\xi^n(2\log(1+f_\omega)))f_\omega\frac{\omega_X^n}{V_{\omega_X}}\\
            =&\frac{1}{A_{\textbf I,k}}\int_X J_{\textbf I}(\hat{H}_k)^n\bar{g}_{\textbf I,\xi}^n(2\log(1+f_\omega)))f_\omega\frac{\omega_X^n}{V_{\omega_X}}+\int_X
            \bar{g}_{\textbf I,\xi}^n(2\log(1+f_\omega)))f_\omega\frac{\omega_X^n}{V_{\omega_X}}\\
            \leq& \frac{1}{A_{\textbf I,k}}\int_X \hat{H}_k J_{\textbf I}(\hat{H}_k) f_\omega\frac{\omega_X^n}{V_{\omega_X}}
            +(2\xi)^n\int_X (e^{\xi \log f_\omega}+C(\xi))(1+e^{2\xi \log f_\omega})e^{\frac{p-1}{2}\log f_\omega}f_\omega\frac{\omega_X^n}{V_{\omega_X}}\\
            +&\int_X (2\xi)^n(e^{\xi \log f_\omega}+C(\xi))(1+e^{2\xi \log f_\omega})f_\omega\frac{\omega_X^n}{V_{\omega_X}}\\
            \leq& C
        \end{aligned}
    \end{equation}
    where we have applied the estimate \refe{young} and \refe{g} in the first inequality and
    chosen $\xi\leq \frac{p-1}{6}$. Therefore, \refe{2} is also uniformly bounded.
    In conclusion, we have that, for some $\xi>0$, $N_{\Phi_{\textbf I,\xi}}(\omega_{\textbf{I},k})=\int_X\Phi_{\textbf I,\xi}(\log(f_{\textbf{I},k}+1))f_{\textbf{I},k}\frac{\omega_X^n}{V_{\omega_X}}$
    is uniformly bounded. \\
    
\noindent \underline{Uniform bound on $N_{\Phi_{\textbf{II}}}(\omega_{\textbf{II},k})$ in \textbf{Case} \textbf{II}.}\\
We choose a function $g_{\textbf{II}}:(0,\infty)\to(0,\infty)$ by 
    $g_{\textbf{II}}(v)= \min\{h(v+a),\log^q(v+a')\}$ where 
    we have chosen $a'=a'(n,q)>0$ such that 
    $\frac{v^{\frac{q-n}{q}}}{\log^q(v+a')}$ is increasing,
    then $\frac{v^{\frac{q-n}{n}}}{g^{q/n}(v)}$ is an increasing function, as $\frac{v^{\frac{q-n}{q}}}{h(v+a)}=\left(\frac{v}{v+a}\right)^{\frac{q-n}{q}}\cdot\frac{(v+a)^{\frac{q-n}{q}}}{h(v+a)}$ is increasing, 
    and 
    \[
        \int_1^\infty \frac{1}{vg_{\textbf{II}}^{1/n}(v)}\mathrm{d}v<\infty.
    \] 
    Define $\bar{g}_{\textbf{II}}(v):=vg_{\textbf{II}}^{1/n}(v)$ and $\bar{g}_{\textbf{II},\xi}(v)=\bar{g}_{\textbf{II}}(\xi v)$.
    Let
    \[
        \Phi_{\textbf{II},\xi}(v):=\bar{g}^n_{\textbf{II},\xi}(v)
    \]
    which is a K-function for each $\xi>0$.
    We are going to uniformly bound 
    $$N_{\Phi_{\textbf{II},\xi}}(\omega_{\textbf{II},k})=\int_X\Phi_{\textbf{II},\xi}(\log(f_{\textbf{II},k}+1))f_{\textbf{II},k}\frac{\omega_X^n}{V_{\omega_X}}$$ 
    for some $\xi>0$.
    Firstly, similar to Case $\textbf I$, by \eqref{ele1} and \eqref{ele2} we have
    \begin{align}
            &\int_X\Phi_{\textbf{II},\xi}(\log(f_{\textbf{II},k}+1))f_{\textbf{II},k}\frac{\omega_X^n}{V_{\omega_X}} \notag\\
            \leq &\int_X (J_{\textbf{II}}(\hat{H}_k)^n/A_{\textbf{II},k}+1)(\bar{g}_{\textbf{II},\xi}^n(
                2\log(2+J_{\textbf{II}}(\hat{H}_k)^n/A_{\textbf{II},k})))f_\omega\frac{\omega_X^n}{V_{\omega_X}}\label{21} \\
            &+\int_X (J_{\textbf{II}}(\hat{H}_k)^n/A_{\textbf{II},k}+1)(
                \bar{g}_{\textbf{II},\xi}^n(2\log(1+f_\omega)))f_\omega\frac{\omega_X^n}{V_{\omega_X}}\label{22}
    \end{align}
    
    To estimate \refe{21}, we note
    \[
        \bar{g}_{\textbf{II},\xi}^n(2\log(2+J_{\textbf{II}}^n(\hat{H}_k)/A_k))\leq
        \bar{g}_{\textbf{II}}^n(\log(1+\hat{H}_k))
    \]
    where we have applied the estimate $A_{\textbf{II},k}\geq1$ and 
    \begin{equation}\label{2log term}
        \log(2+J_{\textbf{II}}^n(v)\leq \log(2+v^{\frac{nq-n^2}{nq-q+n}}/h^{\frac{q}{nq-q+n}}(1))
     \leq \beta\log(1+v), \,\,for\,\,\,any\,\,\, v\geq1
    \end{equation}
    and chosen $\xi\leq \frac{1}{2\beta}$.
    Then by $g_{\textbf{II}}(t)\leq \log^q(v+a')$ we have 
    \begin{equation}\label{2.1}
        \begin{aligned}
            \int_X \bar{g}_{\textbf{II},\xi}^n(
                2\log(&2+J_{\textbf{II}}^n(\hat{H}_k)/A_k))f_\omega\frac{\omega_X^n}{V_{\omega_X}} \leq
            \int_X \bar{g}_{\textbf{II}}^n(\log(1+\hat{H}_k))f_\omega\frac{\omega_X^n}{V_{\omega_X}}\\
            &=\int_X \log^n(1+\hat{H}_k) 
            g_{\textbf{II}}(\log(1+\hat{H}_k))f_\omega\frac{\omega_X^n}{V_{\omega_X}}\\
            &\leq C\int_X \log^{n+1}(1+\hat{H}_k)
            f_\omega\frac{\omega_X^n}{V_{\omega_X}}\leq C
        \end{aligned}
    \end{equation}
    where we have applied 
    the inequality $\log^{r}(v+a')\leq C(r,a') v$
    for any $v\geq1$ and  \refl{int} (see \cite{GPSS1,GPSS2}).
    By $g(v)\leq h(v+a)$ and $g(v)\leq \log^q(v+a')$ we have
    \begin{equation}\label{2.2}
        \begin{aligned}
            &\frac{1}{A_{\textbf{II},k}}\int_X
            J_{\textbf{II}}^n(\hat{H}_k)\bar{g}_{\textbf{II},\xi}^n(
                2\log(2+J_{\textbf{II}}^n(\hat{H}_k)/A_{\textbf{II},k}))f_\omega\frac{\omega_X^n}{V_{\omega_X}}\\
            \leq&\frac{1}{A_{\textbf{II},k}}\int_X
            J_{\textbf{II}}^n(\hat{H}_k)\bar{g}_{\textbf{II}}^n
                (\log(1+\hat{H}_k))f_\omega\frac{\omega_X^n}{V_{\omega_X}}\\
            =&\frac{1}{A_{\textbf{II},k}}\int_X
            \hat{H}_k J_{\textbf{II}}(\hat{H}_k)\frac{
                g_{\textbf{II}}^{\frac{nq-q}{nq-q+n}}(\hat{H}_k)}{
                h^{\frac{nq-q}{nq-q+n}}(\hat{H}_k+a)}
            \frac{
                \log^n(1+\hat{H}_k)g_{\textbf{II}}^{\frac{n}{nq-q+n}}(\log(\hat{H}_k+1))
            }{\hat{H}_k^{\frac{n^2}{nq-q+n}}}
            f_\omega\frac{\omega_X^n}{V_{\omega_X}}\\
            \leq& C.
        \end{aligned}
    \end{equation}
    Combining \refe{2.1} and \refe{2.2} we have that \refe{21} is uniformly bounded.\\
    To estimate \refe{22}, we observe that 
    \begin{equation}\label{2g}
        \begin{aligned}
            \bar{g}_{\textbf{II},\xi}^n(2\log(1+f_\omega))&=(2\xi)^n\log^n(1+f_\omega)g_{\textbf{II}}(2\xi\log(1+f_\omega))\\
            &\leq \log^{q}(1+f_\omega)+C\log^n(1+f_\omega)\\
        \end{aligned}
    \end{equation}
    where we have applied the estimates
    $g_{\textbf{II}}(v)\le \log^q(v+a')\leq v^{q-n}+C$ for any $v \ge0$
    and chosen $\alpha\leq \frac{1}{2}$. 
    We define an increasing function $\eta_{\textbf{II}}:(0,\infty)\to(0,\infty)$ with $\eta_{\textbf{II}}(0)=0$ by 
    \[
        \eta_{\textbf{II}}(z)=\frac{z^{\frac{q-n}{n}}}{g_{\textbf{II}}^{q/n}(\xi z)}\ge\frac{z^{\frac{q-n}{n}}}{g_{\textbf{II}}^{q/n}(z)}
    \]
    Let $w=\eta_{\textbf{II}}(z)$, then 
    \[
        w\ge \frac{z^{q/n-1}}{\log^{q^2/n}(z+a')}\ge z^{\frac{q-n}{(1+\varepsilon)n}}-b_\varepsilon
    \]
    for any $z\ge0$, where $\varepsilon>0$ is to be determined.
    We have
    \begin{equation}\label{inverse eta}
        \eta_{\textbf{II}}^{-1}(w)=z=w^{\frac{n}{q-n}}g^{\frac{q}{q-n}}_{\textbf{II},\xi}(z)\le 
        w^{\frac{n}{q-n}}g_{\textbf{II}}^{\frac{q}{q-n}}(\xi(w+b_\varepsilon)^\frac{(1+\varepsilon)n}{q-n}).
    \end{equation}
    By choosing $c=\exp\{\frac{1}{\xi^{1/(n-1)}}\}$ satisfying
    \[
        \alpha\bar{g}_{\textbf{II}}^n(\log(c+f_\omega))\ge\xi\log^n(c+f_\omega)\ge \log(c+f_\omega),
    \] 
    we have
    \begin{equation}\label{2young1}
        \begin{aligned}
            \bar{g}^n_{\textbf{II},\xi}(2\log(1+f_\omega))
            \eta_{\textbf{II}}(\bar{g}^n_{\textbf{II},\xi}(2\log(1+f_\omega)))
            &\leq\frac{
                \bar{g}_{\textbf{II}}^q(\log(c+f_\omega))
           }{g_{\textbf{II}}^{q/n}(\alpha\bar{g}_{\textbf{II}}^n(\log(c+f_\omega)))}\\
            &\leq \log^q((c+f_\omega))\frac{g_{\textbf{II}}^{q/n}(\log(c+f_\omega))}{
                g_{\textbf{II}}^{q/n}(\log(c+f_\omega))
            }\\
            &= \log^q(c+f_\omega).
        \end{aligned}
    \end{equation}
    
    On the other hand,
    \begin{equation}\label{2young2}
        \begin{aligned}
            J_{\textbf{II}}^n(\hat{H}_k)&\eta_{\textbf{II}}^{-1}(J_{\textbf{II}}^n(\hat{H}_k))\le J_{\textbf{II}}^{\frac{nq}{q-n}}(\hat{H}_k)
            g_{\textbf{II}}^{\frac{q}{q-n}}(\xi(J_{\textbf{II}}^n(\hat{H}_k)+b_\varepsilon)^{\frac{(1+\varepsilon)n}{q-n}})\\
            &\leq J_{\textbf{II}}^{\frac{nq}{q-n}}(\hat{H}_k)g_{\textbf{II}}^{\frac{q}{q-n}}(\xi\mu\hat{H}_k^{\frac{
                (1+\varepsilon)n^2
            }{nq-q+n}})\\
            &\leq J_{\textbf{II}}^{\frac{nq}{q-n}}(\hat{H}_k)h^{\frac{q}{q-n}}(\hat{H}_k)\\
            &=\hat{H}_k J_{\textbf{II}}(\hat{H}_k)
        \end{aligned}
    \end{equation}
    where we have chosen $\varepsilon=\frac{(q-n)(n-1)}{n^2}$ with
    $\frac{(1+\varepsilon)n^2}{nq-q+n}=1$, and $\mu$ satisfying
    \[
        (J_{\textbf{II}}^n(\hat{H}_k)+b_\varepsilon)^{\frac{(1+\varepsilon)n}{q-n}}
    \leq \mu\hat{H}_k^{\frac{(1+\varepsilon) n^2}{nq-q+n}}.
    \]
    and the number $\xi$ with
    $\xi\leq \frac{1}{\mu}$.
    By \refe{2young1}, \refe{2young2} and the generalized Young's inequality with respect to $\eta_{\textbf{II}}$, we have
    \begin{equation}\label{2young}
        \begin{aligned}
            &J_{\textbf{II}}^n(\hat{H}_k)\bar{g}_{\textbf{II},\xi}^n(2\log(1+f_\omega))\\
            \leq& J_{\textbf{II}}^n(\hat{H}_k)\eta_{\textbf{II}}^{-1}(J_{\textbf{II}}^n(\hat{H}_k))
            +\bar{g}_{\textbf{II},\xi}^n(2\log(1+f_\omega))\eta_{\textbf{II}}(\bar{g}_{\textbf{II},\xi}^n(2\log(1+f_\omega)))\\
            \leq&\hat{H}_k J_{\textbf{II}}(\hat{H}_k)+\log^q(c+f_\omega).
        \end{aligned}
    \end{equation}
    Then 
    \begin{equation}
        \begin{aligned}
            \int_X& (J_{\textbf{II}}^n(\hat{H}_k)/A_{\textbf{II},k}+1)(
                \bar{g}_{\textbf{II},\xi}^n(2\log(1+f_\omega)))f_\omega\frac{\omega_X^n}{V_{\omega_X}}\\
            =&\frac{1}{A_{\textbf{II},k}}\int_X J_{\textbf{II}}^n(\hat{H}_k)\bar{g}_\alpha^n(2\log(1+f_\omega))f_\omega\frac{\omega_X^n}{V_{\omega_X}}+\int_X
            \bar{g}_{\textbf{II},\xi}^n(2\log(1+f_\omega))f_\omega\frac{\omega_X^n}{V_{\omega_X}}\\
            \leq& \frac{1}{A_{\textbf{II},k}}\int_X \hat{H}_k J_{\textbf{II}}(\hat{H}_k) f_\omega\frac{\omega_X^n}{V_{\omega_X}}
            +\int_X \log^q(c+f_\omega)f_\omega\frac{\omega_X^n}{V_{\omega_X}}
            +\int_X \log^q(1+f_\omega)f_\omega\frac{\omega_X^n}{V_{\omega_X}}\nonumber\\
            &\,\,\,\,\,+C\int_X \log^n(1+f_\omega)f_\omega\frac{\omega_X^n}{V_{\omega_X}}\\
            \leq& C
        \end{aligned}
    \end{equation}
    where we have applied the estimate \refe{2young} and \refe{2g} in the first inequality. Therefore, the term \eqref{22} is also uniformly bounded. We have proved that $N_{\Phi_{\textbf{II},\xi}}(\omega_{\textbf{II},k})$ is uniformly bounded for some $\xi>0$.\\

     Applying Theorem \ref{Linfty} gives that 
    \begin{equation}
        \sup_X|\psi_{\gamma,k}-\varphi|\leq C_{\gamma}
    \end{equation}
    for some uniform constant $C_{\gamma}$.\\

    For each Case $\gamma=\textbf{I}$ and $\textbf{II}$, we now consider the function
    \begin{equation}
        v_{\gamma,k}:=(\psi_{\gamma,k}-\varphi)-\int_X(\psi_{\gamma,k}-\varphi)\frac{\omega^n}{V_\omega}.
    \end{equation}
    We then calculate the Laplacian of $v_{\gamma,k}$:
    \begin{equation}
        \begin{aligned}
            \Delta_\omega v_{\gamma,k}&=\mathrm{tr}_\omega(\omega_{\gamma,k})-n\\
            &\geq n\left(\frac{\omega^n_{\gamma,k}}{\omega^n}\right)^{1/n}-n\\
            &\geq nB_{\gamma,k}^{-1/n}J_{\gamma}(\hat{H}_k)/A^{1/n}_{\gamma,k}-n.
        \end{aligned}
    \end{equation}
    Applying Green's formula to $v_{\gamma,k}$ and \refl{int}, we have 
    \begin{equation}
        v_{\gamma,k}(x)=-\int_X \mathcal{G}(x,\cdot)\Delta_\omega v_{\gamma,k}\omega^n\leq 
        -C\frac{1}{A_{\gamma,k}^{1/n}} \int_X\mathcal{G}(x,\cdot)J_{\gamma}(\hat{H}_k)\omega^n+C.
    \end{equation}
    Since $v_{\gamma,k}$ is uniformly bounded,  
    we have
    \[
        \int_X\mathcal{G}(x,\cdot)J_{\gamma}(\hat{H}_k)\omega^n\leq C\left(\int_X H_kJ_{\gamma}(\hat{H}_k)
        \omega^n+1\right)^{\frac{1}{n}}
    \]
    In Case $\textbf{I}$,  we observe that 
    \[
        J_{\textbf I}(\hat{H}_k)\leq \frac{\hat{H}_k^{\frac{1}{n-1}}}{\log^{\frac{n}{n-1}}(\hat{H}_k+1)}=:\tilde{J}_{\textbf I}(\hat{H}_k),
    \]
    and 
    \[
        \lim_{k\to\infty}\int_X \mathcal{G}(x,\cdot)\tilde{J}_{\textbf{I}}(\hat{H}_k)\omega^n
        =\int_X \mathcal{G}(x,\cdot)\tilde{J}_{\textbf I}(\mathcal{G}(x,\cdot))\omega^n,
    \]
    \[
        \lim_{k\to\infty}\int_X H_k\tilde{J}_{\textbf I}(\hat{H}_k)\omega^n
        =\int_X \mathcal{G}(x,\cdot)\tilde{J}_{\textbf I}(\mathcal{G}(x,\cdot))\omega^n;
    \]
    while in Case $\textbf{II}$, $J_{\textbf{II}}$ is an increasing function on $[1,\infty)$. 
    We then apply Dominated Convergence Theorem and Monotone Convergence Theorem in Cases $\textbf{I}$ and $\textbf{II}$, respectively, to conclude that, for each $\gamma=\textbf{I}$ and $\textbf{II}$,
    \begin{equation}
            \int_X\mathcal{G}(x,\cdot)J_\gamma(V_\omega\mathcal{G}(x,\cdot))
            \omega^n\leq C\left(\int_X \mathcal{G}(x,\cdot)J_\gamma(V_\omega\mathcal{G}(x,\cdot))
            \omega^n+1\right)^{\frac{1}{n}}
    \end{equation}
    which immediately implies \refe{goal2}. 
    
    Proposition \ref{prop-green-1} is proved.
\end{proof}

\begin{proof}[Proof of Theorems \ref{thmgreen} (i) and \ref{thmgreen-q} (i)]
Applying Proposition \ref{prop-green-1} in Cases $\textbf{I}$ and $\textbf{II}$ with $a=1$ completes the proof of Theorems \ref{thmgreen} (i) and \ref{thmgreen-q} (i), respectively.
\end{proof}

\section{Integral estimate on gradient of Green's function: proof of Theorem \ref{thmgreen} (ii)}\label{sect-nabla}

As an application of the uniform integral estimates obtained in Proposition \ref{prop-green-1}, in this section we will prove uniform intrgral estimates on the gradient of Green's function under both $L^p$ and $L^1(\log L)^q$ conditions on volume density.

\begin{prop}\label{prop-nabla-general}
For each Case $\gamma=\textbf{I}$ and $\textbf{II}$, there is a positive number $\tilde C_{\gamma}$ such that for any $x\in X$,
$$V_\omega^{\tilde k_{\gamma}-1}\int_X\frac{|\nabla_\omega\mathcal G(x,\cdot)|^{\tilde k_{\gamma}}}{j_\gamma^{\tilde k_{\gamma}}(V_\omega\mathcal G(x,\cdot))}\omega^n\le \tilde C_\gamma,$$
where $\tilde k_{\gamma}:=\frac{2k_\gamma}{k_\gamma+1}$.
\end{prop}

By definition, $\tilde k_{\textbf{I}}=\frac{2n}{2n-1}$ and $\tilde k_{\textbf{II}}=\frac{2nq}{2nq-q+n}$.

For the proof of Proposition \ref{prop-nabla-general}, we will need an independent lemma anologous to \cite[Lemma 5.6]{GPSS1} and \cite[Lemma 7.7]{Liuj}.

\begin{lemma}\label{nabla1}
There is a constant $C_h=\int_{\log 2}^\infty\frac{1}{\overline{h}(r)}\mathrm{d}r$ such that for any $x\in X$ and $a\ge1$, we have 
    \begin{equation}
        \int_X \frac{|\nabla_y \mathcal G(x,y)|^2}{
            \mathcal{G}(x,y)j_\gamma(V_\omega\mathcal{G}(x,y))
        }\omega^n(y)\leq (1+a)C_h.
    \end{equation}
\end{lemma}
\begin{proof}
  For each Case $\gamma=\textbf{I}$ and $\textbf{II}$ we consider the function $u_\gamma:(1,\infty)\to(0,\infty)$ which is 
    defined by the following ordinary differential equation
    \begin{equation}
        \begin{cases}
            u_\gamma'(v)=-\frac{1}{(v+a)j_\gamma(v)}\\
            u_\gamma(\infty)=0
        \end{cases}
    \end{equation}
    then by \refe{h}, we have 
    \begin{equation}\label{u<C}
        \begin{aligned}
            u_\gamma(v)=&\int_v^{\infty}\frac{1}{(s+a)j_\gamma(s)}\mathrm{d}s\\
                \le&\int_{\log(v+a)}^\infty\frac{1}{\overline{h}(r)}\mathrm{d}r\\
                \leq& \int_{\log 2}^\infty\frac{1}{\overline{h}(r)}\mathrm{d}r=:C_h<+\infty.
        \end{aligned}
    \end{equation}
    We now fix a point $x\in X$ and let $U_\gamma(y)=u_\gamma(V_\omega\mathcal{G}(x,y))$.
    Applying the Green's formula, we have
    \begin{equation}
        \begin{aligned}
            0&=U_\gamma(x)\\
            &=\frac{1}{V_\omega}\int_X U_\gamma \omega^n+\int_X
            \nabla\mathcal{G}(x,\cdot)\cdot\nabla U_\gamma \omega^n\\
            &=\frac{1}{V_\omega}\int_X U_\gamma \omega^n-\int_X
            \frac{V_\omega|\nabla \mathcal G(x,\cdot)|^2}{
                (V_\omega\mathcal{G}(x,\cdot)+a)
                j_\gamma(V_\omega\mathcal{G}(x,\cdot)}\omega^n
        \end{aligned}
    \end{equation}
    The lemma then follows easily from \refe{u<C} and 
    \[
        \frac{V_\omega}{V_\omega\mathcal{G}(x,\cdot)+a}\ge \frac{1}{(1+a)\mathcal{G}(x,\cdot)}.
    \]
\end{proof}

\begin{proof}[Proof of Proposition \ref{prop-nabla-general}]
Applying H\"older inequality, Lemma \ref{nabla1} and Proposition \ref{prop-green-1} gives
    \begin{align}\label{nablag}
          & \int_X \frac{|\nabla \mathcal G(x,\cdot)|^{\tilde k_\gamma}}{
       j_\gamma^{\tilde k_\gamma}(V_\omega\mathcal{G}(x,\cdot)
    }\omega^n\nonumber\\
    &=\int_X \frac{
        |\nabla \mathcal G(x,\cdot)|^{\tilde k_\gamma}
    }{
        \mathcal{G}(x,\cdot)^{\frac{\tilde k_\gamma}{2}} j_\gamma^{\frac{\tilde k_\gamma}{2}}(V_\omega\mathcal{G}(x,\cdot)
    }\cdot\frac{\mathcal{G}(x,\cdot)^{\frac{\tilde k_\gamma}{2}}}{
       j_\gamma^{\frac{\tilde k_\gamma}{2}}(V_\omega\mathcal{G}(x,\cdot)
    }\omega^n\nonumber\\
    &\leq  \left(\int_X  \frac{|\nabla \mathcal G(x,\cdot)|^2}{
        \mathcal{G}(x,\cdot)j_\gamma(V_\omega\mathcal{G}(x,\cdot))}\omega^n
        \right)^{\frac{\tilde k_\gamma}{2}}\left(
            \int_X \frac{\mathcal{G}(x,\cdot)^{k_\gamma}}{
               j_\gamma^{k_\gamma}(V_\omega\mathcal{G}(x,\cdot))
            }\omega^n
        \right)^{1-\frac{\tilde k_\gamma}{2}}\nonumber\\
    &\leq C_1\cdot C_2V_\omega^{-\frac{k_\gamma-1}{k_\gamma+1}}=CV_\omega^{-(\tilde k_\gamma-1)},
        \end{align}
Proposition \ref{prop-nabla-general} is proved. 
\end{proof}

\begin{proof}[Proof of Theorem \ref{thmgreen} (ii) and Theorem \ref{thmgreen-q} (ii)]
Applying Proposition \ref{prop-nabla-general} in Cases $\textbf{I}$ and $\textbf{II}$ with $a=1$ concludes Theorem \ref{thmgreen} (ii) and Theorem \ref{thmgreen-q} (ii), respectively.
\end{proof}

\section{Volume non-collapsing estimate: proof of Theorem \ref{thmgeom} (a)}\label{sect-volume}

We firstly provide a general version of volume noncollapsing estimate, as an application of Proposition \ref{prop-nabla-general}.

\begin{prop}\label{prop-vol-general}
For each Case $\gamma=\textbf{I}$ and $\textbf{II}$, if $j_\gamma^{\frac{\tilde k_\gamma}{\tilde k_\gamma-1}}$ is concave on $[1,+\infty)$, then there is a positive number $c_\gamma$ such that for any $x\in X$ and $R\in(0,1]$,
\begin{equation}
\underline V_\omega(x,R)j_\gamma^{\frac{\tilde k_\gamma}{\tilde k_\gamma-1}}\left(\frac{\tilde C_0}{\underline V_\omega(x,R)}\right)\ge c_\gamma R^{\frac{\tilde k_\gamma}{\tilde k_\gamma-1}}\nonumber,
\end{equation}
where $\underline{V}_\omega(x,R):=\frac{\mathrm{Vol}_\omega(B_\omega(x,R))}{V_\omega}$, $\tilde C_0:=C_0+2$ and $C_0$ is a number satisfying Theorem \ref{thmgpss} (0).
\end{prop}
By definition, $\frac{\tilde k_{\textbf{I}}}{\tilde k_{\textbf{I}}-1}=2n$ and $\frac{\tilde k_{\textbf{II}}}{\tilde k_{\textbf{II}}-1}=\frac{2nq}{q-n}$.
\begin{proof}
    For a fixed point $x\in X$, we choose a smooth cutoff function $\eta$ with support in 
    $B(x,R)$ satisfying
    \[
        \eta=1\,\,\,\mathrm{on} \,\,\,B\left(x,\frac{R}{2}\right), \,\,\,\,
        \sup_X|\nabla\eta|_\omega\leq \frac{4}{R}.
    \]
    Applying the Green's formula to $\eta$, we have for any $z\in X$
    \begin{equation}\label{cutoff}
        \eta(z)=\frac{1}{[\omega]^n}\int_X \eta\omega^n+
        \int_X \langle \nabla G(z,\cdot),\nabla \eta \rangle_\omega \omega^n.
    \end{equation}
    Take $z\in X\backslash\overline{B(x,R)}$, then 
    \[
        0=\frac{1}{[\omega]^n}\int_X \eta\omega^n+
        \int_X \langle \nabla G(z,\cdot),\nabla \eta \rangle_\omega \omega^n.
    \]
    We have 
    \begin{equation}\label{Non 1}
        \begin{aligned}
            1&=\eta(x)=\frac{1}{[\omega]^n}\int_X \eta\omega^n+
        \int_X \langle \nabla G(x,\cdot),\nabla \eta \rangle_\omega \omega^n\\
        &=-\int_X \langle \nabla G(z,\cdot),\nabla \eta \rangle_\omega \omega^n
        \int_X \langle \nabla G(x,\cdot),\nabla \eta \rangle_\omega \omega^n\\
        &\leq 2\sup_{y\in X}
        \int_X |\nabla G(y,\cdot)|\cdot|\nabla \eta| \omega^n\\
        &\leq \frac{8}{R}\sup_{y\in X}
        \int_{B(x,R)} |\nabla G(y,\cdot)| \omega^n
        \end{aligned}
    \end{equation}
    On the other hand, 
    \begin{equation}\label{Non 2}
        \begin{aligned}
            &\int_{B(x,R)} |\nabla \mathcal G(y,\cdot)| \omega^n\\
            &=\int_{B(x,R)} \frac{|\nabla\mathcal G(y,\cdot)|}{j_\gamma(
               V_\omega \mathcal{G}(y,\cdot))}\cdot
            j_\gamma(
               V_\omega \mathcal{G}(y,\cdot))\omega^n\\
            &\leq \left(
                \int_X \frac{|\nabla\mathcal G(y,\cdot)|^{\tilde k_\gamma}}{
                   j_\gamma^{\tilde k_\gamma}(
               V_\omega \mathcal{G}(y,\cdot))}
                \omega^n
            \right)^{\frac{1}{\tilde k_\gamma}}
            \left(
                \int_{B(x,R)}j_\gamma^{\frac{\tilde k_\gamma}{\tilde k_\gamma-1}}(
               V_\omega \mathcal{G}(y,\cdot))\omega^n
            \right)^{\frac{\tilde k_\gamma-1}{\tilde k_\gamma}}\\
            &\leq CV_\omega^{-\frac{\tilde k_\gamma-1}{\tilde k_\gamma}} \left(
                \int_{B(x,R)}j_\gamma^{\frac{\tilde k_\gamma}{\tilde k_\gamma-1}}(
               V_\omega \mathcal{G}(y,\cdot))\omega^n
            \right)^{\frac{\tilde k_\gamma-1}{\tilde k_\gamma}}\\
            &=C\underline{V}(x,R)^{\frac{\tilde k_\gamma-1}{\tilde k_\gamma}} \left(
                \frac{1}{Vol_\omega(B(x,R))}\int_{B(x,R)}j_\gamma^{\frac{\tilde k_\gamma}{\tilde k_\gamma-1}}(
               V_\omega \mathcal{G}(y,\cdot))\omega^n
            \right)^{\frac{\tilde k_\gamma-1}{\tilde k_\gamma}}\\
            &\le C\underline{V}(x,R)^{\frac{\tilde k_\gamma-1}{\tilde k_\gamma}} j_\gamma\left(\frac{C_0+2}{\underline V_\omega(x,R)}\right),
        \end{aligned}
    \end{equation}
where in the first inequality we have used H\"older inequality and in the last inequality we have applied Jensen inequality with respect to the concave function $j_\gamma^{\frac{\tilde k_\gamma}{\tilde k_\gamma-1}}$.

Combining \eqref{Non 1} and \eqref{Non 2} finishes the proof of Proposition \ref{prop-vol-general}.
\end{proof}

Next we check concavity of $j_\gamma^{\frac{\tilde k_\gamma}{\tilde k_\gamma-1}}$ in a useful setting.
In the remaining part of this section, we assume $h$ is of the following specific form: arbitrarily given $k\ge1$ and $b>1$, set $l_0(v)=1$, $l_1(v)=\log(v+b)$, and for $2\le d\le k$, $l_d(v)=l_{1}(l_{d-1}(v))=\log(l_{d-1}(v)+b)$; the number $b=b(k)$ is large with $l_{d}(1)\ge1$ for each $d=1,...,k$. Then we assume 
\begin{align}\label{h-form}
h(v)=l_0(v)l_1^n(v)\cdots l_{k-1}^n(v)\cdot l_{k}^r(v),
\end{align}
where $k\ge1$ and $r>n$. In Case $\textbf{II}$ we further increase $b=b(k,n,q)$ such that $\frac{v^{\frac{q-n}{q}}}{h(v)}$ is increasing on $[1,\infty)$.

\begin{lemma}\label{lem-concave}
For $h$ given by \eqref{h-form}, there is a positive number $a_0=a_0(n,k,b,r)$ such that for $a\ge a_0$,
\begin{itemize}
\item[(1)] $j_\gamma^{\frac{\tilde k_\gamma}{\tilde k_\gamma-1}}$ is concave on $[1,\infty)$;
\item[(2)] the function $zj_\gamma^{\frac{\tilde k_\gamma}{\tilde k_\gamma-1}}\left(\frac{\tilde C_0}{z}\right)$ is increasing for $z\in(0,1)$.
\end{itemize}
\end{lemma}
\begin{proof}
Given the explicit definition of the involved functions, this lemma can be checked usig elementary arguments. Let's just look at Case $\textbf{I}$, as Case $\textbf{II}$ can be handled identically. Firstly, for convenience we set

\begin{align}\label{Gamma-formula}
\Gamma(v)=\Gamma_{k,r}(v):=\bar h^{2n}(v)=v^{2n}l_1^{2n}(v)\cdots l_{k-1}^{2n}(v)\cdot l_{k}^{2r}(v),
\end{align}
then 
$$j_{\textbf{I}}^{\frac{\tilde k_{\textbf I}}{\tilde k_{\textbf{I}}-1}}(v)=\Gamma(\log(v+a))$$
By direct computations, we have
$$\Gamma'(v)=2nv^{2n-1}l_1^{2n}(v)\cdots l_{k-1}^{2n}(v)l_k^{2r}(v)\left(1+O(l_{k-1}(v)^{-1})\right),$$
and 
$$\Gamma''(w)=2n(2n-1)v^{2n-2}l_1^{2n}(v)\cdots l_{k-1}^{2n}(v)l_k^{2r}(v)\left(1+O(l_{k-1}(v)^{-1})\right),$$
where a function $A(v)=O(l_{k-1}(v)^{-1})$ means that there is a constant $C(n,k,r)$ and a number $\hat v=\hat v(n,k,r)\ge1$ such that $|A(v)|\le\frac{C(n,k,r)}{l_{k-1}(v)}$ for any $v\ge \hat v$. Therefore, there is a number $v_0=v_0(n,k,r)\ge1$ such that
$$\Gamma''(v)-\Gamma'(v)\le0,\,\,\,\forall \,v\ge v_0.$$
Finally, we conclude that for $a\ge a_0=a_0(n,k,r):=e^{v_0}$,
$$(j_{\textbf{I}}^{\frac{\tilde k_{\textbf I}}{\tilde k_{\textbf{I}}-1}})''(v)=(\Gamma''(\log(v+a))-\Gamma'(\log(v+a)))(v+a)^{-2}\le0,\,\,\,\forall \,v\ge1,$$
namely, $j_{\textbf{I}}^{\frac{\tilde k_{\textbf I}}{\tilde k_{\textbf{I}}-1}}$ is concave on $[1,\infty)$. Item (1) is proved.

To see item (2), we set $\Theta(z):=zj_{\textbf{I}}^{\frac{\tilde k_{\textbf I}}{\tilde k_{\textbf{I}}-1}}\left(\frac{\tilde C_0}{z}\right)$, then using the explicit definitions of the involved functions we easily see that
$$\Theta'(z)=j_{\textbf{I}}^{\frac{\tilde k_{\textbf I}}{\tilde k_{\textbf{I}}-1}}\left(\frac{\tilde C_0}{z}\right)\left(1+O\left(\frac{1}{l_{k-1}(\log(\tilde C_0+a)+b)}\right)\right),$$
where a function $B(z)=O\left(\frac{1}{l_{k-1}(\log(\tilde C_0+a)+b)}\right)$ means that there is a constant $\tilde C(n,k,r)$ such that $|B(z)|\le\frac{C(n,k,r)}{l_{k-1}(\log(\tilde C_0+a)+b)}$ for any $z\in(0,1)$. Therefore, after possibly increasing $a_0=a_0(n,k,r)$, we have that for $a\ge a_0$,
$$\Theta'(z)>0,\,\,\,\forall\,z\in(0,1),$$
which implies the item (2).
\end{proof}

Now we are ready to prove a general version of volume non-collapsing estimate.

\begin{prop}\label{prop-vol-explicit}
In Case $\textbf{I}$, assume the setup of Lemma \ref{lem-concave} with $\Gamma_{k,r}$ defined in \eqref{Gamma-formula}. Then for any $a\ge a_0$, there is a number $c=c_{\textbf I}(a,k,r)$ such that for any $x\in X$ and any $R\in(0,1)$,
$$\underline V_\omega(x,R)\ge\frac{cR^{2n}}{\Gamma_{k,r}(-\log R+1)}.$$
\end{prop}

\begin{proof}
Combining Proposition \ref{prop-vol-general} and Lemma \ref{lem-concave}, we see that 
\begin{equation}
        \underline{V}_\omega(x,R)
      j_{\textbf{I}}^{\frac{\tilde k_{\textbf I}}{\tilde k_{\textbf{I}}-1}}\left(\frac{\tilde C_0}{\underline{V}_\omega(x,R)}\right)\ge \delta_0 R^{2n}.
      \end{equation}
      Then this proposition directly follows from the monotonicity of $\Theta(z)=zj_{\textbf{I}}^{\frac{\tilde k_{\textbf I}}{\tilde k_{\textbf{I}}-1}}\left(\frac{\tilde C_0}{z}\right)=z\Gamma\left(\log(\frac{\tilde C_0}{z}+a)\right)$ and the following\\ 
\underline{Claim}: there is a number $c=c_{\textbf I}(a,k,r)$ such that for any $R\in(0,1)$,
$$\Theta\left(\frac{cR^{2n}}{\Gamma(-\log R+1)}\right)\le\delta_0 R^{2n},$$
which is equivalent to
$$c\Gamma\left(\log\left(\frac{\tilde C_0\Gamma(-\log R+1)}{cR^{2n}}+a\right)\right)\le\delta_0\Gamma(-\log R+1).$$
Given the explicit definition of $\Gamma$ in \eqref{Gamma-formula}, one checks by using elementary arguments that the last inequality holds for sufficiently small $c=c_{\textbf I}(a,k,r)$, as required.

Proposition \ref{prop-vol-explicit} is proved.
\end{proof}

\begin{rem}
In the above Proposition \ref{prop-vol-explicit}, considering its geometric applications, we have restricted the discussions to Case $\textbf{I}$ (i.e. the $L^p$ volume density case). For Case $\textbf{II}$ (i.e. the $L^1(\log L)^q$ volume density case), an anolog can be carried out using identical arguments; more precisely, given $h$ of the form \eqref{h-form} with parameters $k\ge2$ and $r>n$, then in Case $\textbf{II}$, for sufficiently large $a>1$, there is a number $c=c_{\textbf{II}}(a,k,r)$ such that for any $x\in X$ and any $R\in(0,1)$,
$$\underline V_\omega(x,R)\ge\frac{cR^{\frac{2nq}{q-n}}}{\Gamma_{k-1,r}^{\frac{q}{q-n}}(-\log R+1)}.$$
\end{rem}

\begin{proof}[Proof of Theorem \ref{thmgeom} (a)]
This immediately follows from Proposition \ref{prop-vol-explicit}.
\end{proof}

\section{Sobolev-type inequality: proof of Theorem \ref{thmgeom} (b)}\label{sect-sobolev}
In this section, we fix a K\"ahler metric satisfying the assumption in Theorem \ref{thmgeom}. The main result in this section is the following family of Sobolev-type inequalities, which contains Theorem \ref{thmgeom} (b) as the special case of $d=1$.

\begin{theorem}\label{Sobolev1}
For any $1\leq d<n,q=\frac{n}{n-d}$ and $q'>\frac{2nd}{n-d}$, there exist constants $\theta,C>0$ depending on $n,\alpha,A,p,K,d,q'$ such that
    \begin{equation}
        \int_X N\left(\frac{\theta |u-\bar{u}|^{2d}}{\int_X|\nabla u|^{2d}\frac{\omega^n}{V_\omega}}\right)\frac{\omega^n}{V_\omega}\leq C
    \end{equation}
    for all $u\in W^{1,2d}(X)$ where $\bar{u}=\frac{1}{V_\omega}\int_X u\omega^n$ and $N(t)=\frac{t^q}{\log^{q'}(t+2)}$.
\end{theorem}

To prove this theorem, we need the following lemma, which interpolates items (i) and (ii) in Theorem \ref{thmgreen}.

\begin{lemma}\label{nabla2}
    For any $s\in[0,1]$, there exists a constant $C=C(n,\alpha,A,p,K,h,s)$ such that 
    \begin{equation}
        \sup_{x\in X}\int_X \frac{|\nabla_y \mathcal{G}(x,y)|^{2s}}{\mathcal{G}(x,y)^{\frac{2ns-s-n}{n-1}}\overline{h}^{\frac{n-s}{n-1}}(\log(\mathcal{G}(x,y)V_\omega+1))}\omega^n(y)\leq C V_\omega^{\frac{s-1}{n-1}}.
    \end{equation} 
    where $h$ is the same as in Theorem \ref{thmgreen}.
\end{lemma}
\begin{proof}
    For fixed $x\in X$, we regard $\mathcal{G}(y):=\mathcal{G}(x,y)$ as a function of $y$. The $s=0,1$ cases are exactly the conclusions in Theorem \ref{thmgreen}. For $s\in(0,1)$, applying H\"older's inequality gives
    \begin{equation}
        \begin{aligned}
            &\int_X \frac{|\nabla \mathcal{G}|^{2s}}{\mathcal{G}^{\frac{2ns-s-n}{n-1}}\overline{h}^{\frac{n-s}{n-1}}(\log(\mathcal{G}V_\omega+1))}\omega^n=\int_X\frac{|\nabla\mathcal{G}|^{2s}}{\mathcal{G}^s\overline{h}^s(\log(\mathcal{G}V_\omega+1))}\frac{\mathcal{G}^{\frac{n-ns}{n-1}}}{\overline{h}^{\frac{n-ns}{n-1}}(\log(\mathcal{G}V_\omega+1))}\omega^n\\
            &\leq \left(\int_X\frac{|\nabla\mathcal{G}|^2}{\mathcal{G}\overline{h}(\log(\mathcal{G}V_\omega+1))}\omega^n\right)^s\left(\int_X \frac{\mathcal{G}^{\frac{n}{n-1}}}{\overline{h}^{\frac{n}{n-1}}(\log(\mathcal{G}V_\omega+1))}\omega^n\right)^{1-s}\\
            &\leq CV_\omega^{\frac{s-1}{n-1}}.\nonumber
        \end{aligned}
    \end{equation}     
    Lemma \ref{nabla2} is proved.
\end{proof}

\begin{proof}[Proof of Theorem \ref{Sobolev1}]
    Without loss of generality, we may assume $u\in C^1(X)$. Let $s\in (1/2,1]$ to be determined. We apply the Green's formula associated to the metric $\omega$ to obtain
    \[
        u(x)=\frac{1}{V_\omega}\int_X u\omega^n+\int_X\langle\nabla\mathcal{G}(x,\cdot),\nabla u\rangle\omega^n
    \]
    for any $x\in X$. We regard $\mathcal{G}(y):=\mathcal{G}(x,y)$ as a function of $y$. By the triangle inequality and H\"older's inequality this gives
    \begin{equation}\label{u(x)}
        \begin{aligned}
            &|u(x)-\bar{u}|\leq \int_X |\nabla\mathcal{G}(x,\cdot)||\nabla u|\omega^n\\
          & =\int_X \frac{|\nabla\mathcal{G}|}{\mathcal{G}^{\frac{2ns-n-s}{2s(n-1)}}\overline{h}^{\frac{n-s}{2s(n-1)}}(\log(V_\omega\mathcal{G}+1))}\mathcal{G}^{\frac{2ns-n-s}{2s(n-1)}}\overline{h}^{\frac{n-s}{2s(n-1)}}(\log(V_\omega\mathcal{G}+1))|\nabla u|\omega^n\\
            &\leq\left(\int_X \frac{|\nabla\mathcal{G}|^{2s}}{\mathcal{G}^{\frac{2ns-n-s}{n-1}}\overline{h}^{\frac{n-s}{n-1}}(\log(V_\omega\mathcal{G}+1))}\omega^n\right)^\frac{1}{2s}
            \left(\int_X\mathcal{G}^{\frac{2ns-n-s}{(n-1)(2s-1)}}\overline{h}^{\frac{n-s}{(n-1)(2s-1)}}(\log(V_\omega\mathcal{G}+1))|\nabla u|^{\frac{2s}{2s-1}}\omega^n\right)^{1-\frac{1}{2s}}\\
            &\leq C\left(\int_X(V_\omega\mathcal{G})^{\frac{2ns-n-s}{(n-1)(2s-1)}}\overline{h}^{\frac{n-s}{(n-1)(2s-1)}}(\log(V_\omega\mathcal{G}+1))|\nabla u|^{\frac{2s}{2s-1}}\frac{\omega^n}{V_\omega}\right)^{1-\frac{1}{2s}},\\
        \end{aligned}
    \end{equation}
where in the last inequality we have applied \refl{nabla2}. Now let $s=\frac{d}{2d-1}\in(\frac{n}{2n-1},1]$, $\bar{h}(t)=t^r$ where $r>1$ is to be determined, $\psi(t)=t^{\frac{2ns-n-s}{(n-1)(2s-1)}}\log^{\frac{(n-s)r}{(n-1)(2s-1)}}(t+1)$ and $f(t)=\frac{t^{\frac{n}{n-1}}}{\log^{\frac{n}{n-1}m}(t+a)}$ where $m>1$ is to be determined and $a>0$ is a constant depending on $n,m$ such that $f'(t)\ge0(t>0)$. Then
    \begin{equation}
        \frac{\theta|u(x)-\bar{u}|^{\frac{2s}{2s-1}}}{\int_X|\nabla u|^{\frac{2s}{2s-1}}\frac{\omega^n}{V_\omega}}\leq \frac{\hat C^{-1}}{\int_X|\nabla u|^{\frac{2s}{2s-1}}\frac{\omega^n}{V_\omega}}\left(\int_X \psi\circ f^{-1}(f(V_\omega\mathcal{G}))|\nabla u|^{\frac{2s}{2s-1}}\frac{\omega^n}{V_\omega}\right)
    \end{equation}
    where we have chosen $\theta=C^{-\frac{2s}{2s-1}}\cdot \hat C^{-1}$, and $\hat C$ is the number satisfying \eqref{hatC}. By directly computing, 
    \begin{equation}\label{hatC}
        \psi\circ f^{-1}(t)\leq \hat C t^{\frac{2sn-n-s}{2sn-n}}\log^{\frac{(n-s)r+(2ns-n-s)m}{(n-1)(2s-1)}}(t+1)=:\hat C g(t), (t\ge f(1))
    \end{equation}
    and $g''(t)\le 0$ when $t$ is larger than a constant $b$ which can be calculated directly. By Jensen's inequality, we have
    \begin{equation}
        \begin{aligned}
            \frac{\hat C^{-1}}{\int_X|\nabla u|^{\frac{2s}{2s-1}}\frac{\omega^n}{V_\omega}}&\int_X \psi\circ f^{-1}(f(V_\omega\mathcal{G}))|\nabla u|^{\frac{2s}{2s-1}}\frac{\omega^n}{V_\omega}\\
            \leq&\frac{1}{\int_X|\nabla u|^{\frac{2s}{2s-1}}\frac{\omega^n}{V_\omega}}\int_X g(f(V_\omega\mathcal{G})+b)|\nabla u|^{\frac{2s}{2s-1}}\frac{\omega^n}{V_\omega} \\
            \leq &g\left(\frac{1}{\int_X|\nabla u|^{\frac{2s}{2s-1}}\frac{\omega^n}{V_\omega}}\int_X f(V_\omega\mathcal{G}))|\nabla u|^{\frac{2s}{2s-1}}\frac{\omega^n}{V_\omega}+b\right)
        \end{aligned}
    \end{equation}
    By directly computing,$g^{-1}(t)\ge \frac{t^q}{\log^{q'}(t+3)}-C=:N(t)-C$, where we have chosen $q=\frac{2sn-n}{2sn-n-s}, q'=n\frac{nr-sr+2nsm-nm-sm}{(n-1)(2ns-n-s)}$. Since $s=\frac{d}{2d-1}$ and $m,r>1$ we have $q=\frac{n}{n-d}$ and $q'>\frac{2nd}{n-d}$. Then 
    \begin{equation}\label{N(x)}
        \begin{aligned}
            &N\left( \frac{\theta |u(x)-\bar{u}|^{\frac{2s}{2s-1}}}{\int_X|\nabla u|^{\frac{2s}{2s-1}}\frac{\omega^n}{V_\omega}}\right)\leq g^{-1}\left( \frac{\theta|u(x)-\bar{u}|^{\frac{2s}{2s-1}}}{\int_X|\nabla u|^{\frac{2s}{2s-1}}\frac{\omega^n}{V_\omega}}\right)+C\\
            &\leq \frac{1}{\int_X|\nabla u|^{\frac{2s}{2s-1}}\frac{\omega^n}{V_\omega}}\int_X f(V_\omega\mathcal{G})|\nabla u|^{\frac{2s}{2s-1}}\frac{\omega^n}{V_\omega}+C.
        \end{aligned}
    \end{equation}
    Integrating \refe{N(x)} over $x\in X$ against $\frac{\omega^n}{V_\omega}$ and by the symmetry of $\mathcal{G}(x,y)$, we have 
    \begin{equation}
        \begin{aligned}
            &\int_{x\in X} N\left( \frac{\theta |u(x)-\bar{u}|^{\frac{2s}{2s-1}}}{\int_X|\nabla u|^{\frac{2s}{2s-1}}\frac{\omega^n}{V_\omega}}\right)\frac{\omega^n(x)}{V_\omega}\\
            &\leq\frac{1}{\int_X|\nabla u|^{\frac{2s}{2s-1}}\frac{\omega^n}{V_\omega}}\int_{x\in X}\int_{y\in X} f(V_\omega\mathcal{G}(x,y))|\nabla u(y)|^{\frac{2s}{2s-1}}\frac{\omega^n(y)}{V_\omega}\frac{\omega^n(x)}{V_\omega}+C\\
            &=\frac{1}{\int_X|\nabla u|^{\frac{2s}{2s-1}}\frac{\omega^n}{V_\omega}}\int_{y\in X}|\nabla u(y)|^{\frac{2s}{2s-1}}\frac{\omega^n(y)}{V_\omega}\int_{x\in X} f(V_\omega\mathcal{G}(x,y))\frac{\omega^n(x)}{V_\omega}+C\\
            &\leq C,
        \end{aligned}
    \end{equation}
    where in the last step we have applied Theorem \ref{thmgreen} (i) and the symmetry of $\mathcal G$.
    Noting $d=\frac{s}{2s-1}$, we have completed the proof.
\end{proof}

\begin{rem}
Since, for any given $1\le d<n$, $N(v)\ge c(n,d,p') v^{p'}, \forall v\ge1$, for any $0<p'<\frac{n}{n-d}$, this theorem strengthens \cite[Theorem 4.1]{GPSS3} using a different method without involving Riesz-Thorin Interpolation Theorem.
\end{rem}

\begin{rem}
Basing on Theorem \ref{Sobolev1}, similar to Remark \ref{rem-scale} one can also derive a scale-invariant version of Sobolev-type inequality.
\end{rem}

Note that the critical case $d=n$ is not contained in Theorem \ref{Sobolev1} (nor in \cite[Theorem 4.1]{GPSS3}). In Appendix \ref{app-mto} we prove for this case a version of Moser-Trudinger-Onofri type inequality, which may be of independent interest (but is not used in this paper).

\section{Heat kernel estimate: proof of Theorem \ref{thmgeom} (c)}\label{sect-heat}
With the Sobolev type inequality in \reft{Sobolev1}, we now prove the heat kernel upper bound on diagonal by adapting Cheng-Li's method \cite{CL}. As before, we fix a K\"ahler metric $\omega$ satisfying the assumption in Theorem \ref{thmgeom}, and simply write $H=H_\omega$. Firstly we set
\[\tilde{H}(x,y,t):=H(x,y,t)-\frac{1}{V_\omega}\]
which satisfies $\int_X \tilde{H}(x,y,t)\omega^n(y)=0$ for any $t>0$. Recall that $\tilde{H}$ satisfies the following semi-group property
    \begin{equation}
        \tilde{H}(x,z,t)=\int_X \tilde{H}(x,y,s)\tilde{H}(y,z,t-s)\omega^n(y)
    \end{equation}
    for all $s\in(0,t)$. In particular, setting $s=t/2$ and $x=z$ gives the equation
\[\tilde{H}(x,x,t)=\int_X \tilde{H}(x,y,t/2)^2\omega^n(y).\]
Then by taking $t$-derivative and using the heat equation and Stokes's theorem, we have
\begin{equation}\label{linterpolating}
    \frac{\partial}{\partial t}\tilde{H}(x,x,t)=-\int_X|\nabla_y\tilde{H}(x,y,t/2)|^2_\omega \omega^n(y).
\end{equation}
To proceed, we will need a key observation as follows.
\begin{lemma}
    Let $N(t)=\frac{t^{\frac{n}{n-1}}}{\log^{q'}(t+3)}$, where $q'>\frac{2n}{n-1}$. Then for all function $v$ of finite $\Vert v\Vert_{L^1}$ and $\Vert N(v^2)\Vert_{L^1}$, there holds
    \begin{equation}\label{interpolating}
        \Vert v\Vert_{L^2}\leq C(n,q') \Vert v\Vert^{\frac{1}{n+1}}_{L^1}\Vert N(v^2)\Vert ^{\frac{n-1}{2(n+1)}}_{L^1}\log^{\frac{(n-1)q'}{2(n+1)}}\left(c+\frac{\Vert N(v^2)\Vert_{L^1}}{\Vert v\Vert_{L^1}}\right)
    \end{equation}
     where $\Vert \cdot \Vert_{L^k}=\Vert \cdot \Vert_{L^k(X,\frac{\omega^n}{V_\omega})}$ and $c>2$ is a constant such that $\frac{t^{\frac{1}{n-1}}}{\log^{q'}(t+c)}$ is increasing when $t>0$.
\end{lemma}
\begin{proof}
    For any fixed $\lambda>0$, if $|v|\le \lambda$, then $v^2\leq\lambda|v|$ and if $|v|\ge\lambda$, then $|v|^2=N(v^2)\frac{v^2}{N(v^2)}\le N(v^2)\frac{\log^{q'}(v^2+c)}{(v^2)^{\frac{1}{n-1}}}\leq N(v^2)\log^{q'}\frac{\log^{q'}(\lambda^2+c)}{\lambda^{\frac{2}{n-1}}}$. We have
    \begin{equation}\label{v^2}
        \begin{aligned}
            &\int_X v^2 \frac{\omega^n}{V_\omega}\le \int_{|v|\le\lambda} v^2 \frac{\omega^n}{V_\omega}+\int_{|v|\ge\lambda} v^2 \frac{\omega^n}{V_\omega}\\
            &\le \lambda\Vert v\Vert_{L^1}+\frac{\log^{q'}(\lambda^2+c)}{\lambda^{\frac{2}{n-1}}}\Vert N(v^2)\Vert_{L^1}.
        \end{aligned}
    \end{equation}
    for all $\lambda>0$. We choose $\lambda$ such that $\lambda\Vert v\Vert_{L^1}=\frac{\log^{q'}(\lambda^2+c)}{\lambda^{\frac{2}{n-1}}}\Vert N(v^2)\Vert_{L^1}$. Set $I=\frac{\Vert N(v^2)\Vert_{L^1}}{\Vert v\Vert_{L^1}}$, then 
    \[\frac{\lambda^{\frac{n+1}{n-1}}}{\log^{q'}(\lambda^2+c)}=I\ge\lambda-C\]
    where $C>0$ is a constant. We have
    \[\lambda=I^\frac{n-1}{n+1}\log^{\frac{n-1}{n+1}q'}(\lambda^2+c)\le C I^\frac{n-1}{n+1}\log^{\frac{n-1}{n+1}q'}(I+c).\]
    From \refe{v^2} 
    \begin{equation}
        \Vert v\Vert^2_{L^2}\leq 2\lambda\Vert v\Vert_{L^1}\leq C\Vert v\Vert_{L^1}I^{\frac{n-1}{n+1}}\log^{\frac{(n-1)q'}{n+1}}(I+c)
    \end{equation}
    which implies \refe{interpolating}.
\end{proof}

\begin{theorem}
    Let $q'>\frac{2n}{n-1}$ be a constant. There exists a positive constant $C=C(n,\alpha,A,p,K,q')$ such that for any $x\in X$ 
    \begin{equation}\label{H(x,x,t)}
        H(x,x,t)\le \frac{1}{V_\omega}+\frac{C}{V_\omega}t^{-n}\log^{q'(n-1)}(3+\frac{1}{t}).
    \end{equation}
\end{theorem}
\begin{proof}
    For a fixed $x\in X$ and $t>0$, we consider the function $u(y):=\tilde{H}(x,y,t/2).$ By the definition of $\tilde{H}(x,y,t/2)$, we have $\bar{u}=\frac{1}{V_\omega}\int_X u\omega^n=0$ and $\Vert u\Vert_{L^1{(X,\frac{\omega^n}{V_\omega})}}\le \frac{2}{V_\omega}$. Set $v(y)=\frac{\theta^\frac{1}{2}|u|}{\Vert\nabla u\Vert_{L^2(X,\frac{\omega^n}{V_\omega})}}$, where $\theta$ is the constant in \reft{Sobolev1}. By \reft{Sobolev1} in the case $d=1$, we have
    \[\int_X N(v^2)\frac{\omega^n}{V_\omega}\leq C\]
    where $N(t)=\frac{t^{\frac{n}{n-1}}}{\log^{q'}(t+3)}$, $q'>\frac{2n}{n-1}$. Applying \refl{linterpolating}, we get
    \begin{equation}
        \begin{aligned}
            \Vert v\Vert_{L^2}\leq C \Vert v\Vert_{L^1}^{\frac{1}{n+1}}\log^{\frac{(n-1)q'}{2(n+1)}}\left(d+\frac{1}{\Vert v\Vert_{L^1}}\right).
        \end{aligned}
    \end{equation}
    where $d>c$ satisfying $t^{\frac{1}{n+1}}\log^{\frac{(n-1)q'}{2(n+1)}}(d+\frac{1}{t})$ is increasing when $t>0$. By the definition of $v$ and properties of the heat kernel recalled above, we know
    \[\Vert v\Vert_{L^1}\le\frac{C}{\Vert\nabla u\Vert_{L^2}V_\omega}=C(-V_\omega\frac{\partial}{\partial t}\tilde{H}(x,x,t))^{-\frac{1}{2}}\]
    and
    \[\Vert v\Vert_{L^2}=\theta^{1/2}\tilde{H}(x,x,t)^{1/2}/(-\frac{\partial}{\partial t}\tilde{H}(x,x,t))^{1/2}.\]
    Then we have 
    \begin{equation}
        V_\omega\tilde{H}(x,x,t) \le C(-V_\omega\frac{\partial}{\partial t}\tilde{H}(x,x,t))^{\frac{n}{n+1}}\log^{\frac{q'(n-1)}{n+1}}(d+1-V_\omega\frac{\partial}{\partial t}\tilde{H}(x,x,t))
    \end{equation}
    which implies that 
    \begin{equation}
        -V_\omega \frac{\partial}{\partial t}\tilde{H}(x,x,t)\ge C'(V_\omega\tilde{H}(x,x,t))^{\frac{n+1}{n}}\log^{-q'\frac{n-1}{n}}(d'+V_\omega\tilde{H}(x,x,t)).
    \end{equation}
    Noting that $\tilde{H}(x,x,t)\to+\infty$ as $t\to0^+$, we have 
    \begin{equation}
        V_\omega\tilde{H}(x,x,t)\leq Ct^{-n}\log^{(n-1)q'}(3+\frac{1}{t}).
    \end{equation}
    which implies the upper bound of $H(x,x,t)$:
    \[ H(x,x,t)\le \frac{1}{V_\omega}+\frac{C}{V_\omega}t^{-n}\log^{q'(n-1)}(3+\frac{1}{t}).\]
\end{proof}

\section{Eigenvalue and eigenfunction estimates: proof of Theorem \ref{thmgeom} (d-e)}\label{sect-eigen}
Let $\lambda_k$ and $\phi_k$ are eigenvalues and eigenfunctions $-\Delta_\omega$ as chosen in the introduction section. 
\begin{theorem}
Let $q'>\frac{2n}{n-1}$ be a constant. There exists $c=c(n,\alpha,A,p,K,q')>0$ such that for any $k\in\mathbb{N}$,
    \[\lambda_k\ge c \frac{k^{1/n}}{\log^{\frac{n-1}{n}q'}(2+k)}.\]
\end{theorem}
\begin{proof}
It is a classical fact that $\{\phi_k\}$ is an orthonormal basis of $L^2(X,\omega^n)$, $\phi_0=1/\sqrt{V_\omega}$ and the heat kernel $H(x,y,t)$ can be expanded as
    \begin{equation}
        H(x,y,t)=\sum_{j=0}^{\infty}e^{-\lambda_j t}\phi_j(x)\phi_j(y)=\frac{1}{V_\omega}+\sum_{j=1}^{\infty}e^{-\lambda_j t}\phi_j(x)\phi_j(y).\nonumber
    \end{equation}
    The upper bound \refe{H(x,x,t)} implies that
    \begin{equation}\label{expand hk}
        \sum_{j=1}^{\infty}e^{-\lambda_j t}\phi_j(x)^2\leq \frac{C}{V_\omega}t^{-n}\log^{(n-1)q'}(2+\frac{1}{t}).
    \end{equation}
    Integrating \refe{expand hk} over $X$ gives
    \begin{equation}\label{eigenvalues}
        \sum_{j=1}^{\infty}e^{-\lambda_jt}\le Ct^{-n}\log^{(n-1)q'}(2+\frac{1}{t})
    \end{equation}
    For any given $k\ge1$, setting $t=1/\lambda_k$ in \refe{eigenvalues} implies
    \[C\lambda_k^n\log^{(n-1)q'}(3+\lambda_k)\ge\sum_{j=1}^{\infty}e^{-\lambda_j/\lambda_k}\ge\sum_{j=1}^{k}e^{-\lambda_j/\lambda_k}\ge e^{-1}k.\]
    Hence we conclude that there is a positive constant $c$ such that
    \[\lambda_k\ge c\frac{k^{1/n}}{\log^{\frac{n-1}{n}q'}(3+k)}.\]
\end{proof}

\begin{theorem}
Let $q'>\frac{2n}{n-1}$ be a constant. There exists $C=C(n,\alpha,A,p,K,q')>0$ such that for any $k\in\mathbb{N}$,
    \[|\phi_k|_{L^\infty(X)}^2\le CV_\omega^{-1} \lambda_k^{n}\log^{\frac{n-1}{n}q'}(3+\lambda_k).\]
\end{theorem}
\begin{proof}
Again using \eqref{expand hk} with $t=\frac{1}{\lambda_k}$ we have, for any $x\in X$,
$$\phi_k(x)^2\le e\sum_{j=1}^{\infty}e^{-\lambda_j\cdot\frac{1}{\lambda_k}}\phi_j(x)^2\leq \frac{C}{V_\omega}\lambda_k^n\log^{\frac{n-1}{n}q'}(3+\lambda_k),\,\,\forall x\in X,$$
which finishes the proof immediately.
\end{proof}

\section{Applications to the K\"ahler-Ricci flow}\label{sect-vol-pf}
We are ready to prove Theorems \ref{thm-infinite} and \ref{thm-finite} on the K\"ahler-Ricci flow.

\begin{proof}[Proof of Theorem \ref{thm-infinite}]
Let $\omega(t)$, $t\in[0,\infty)$ be the long-time solution to \eqref{KRF1} starting from an initial K\"ahler metric $\omega_0$ on an $n$-dimensional minimal K\"ahler manifold. Fix a smooth closed real $(1,1)$-form $\chi$ representing $2\pi c_1(K_X)$ and set $\omega_t:=e^{-t}\omega_0+(1-e^{-t})\chi$. Then it is well-known that $[\omega(t)]=[\omega_t]$ and hence $\omega_t\le A\omega_0$ for some positive number $A=A(X,\omega_0)$, which is independent in $t$. Then by Tian's $\alpha$-invariant \cite{Ti}, we can choose two positive numbers $\alpha=\alpha(X,\omega_0)$ and $A'=A'(X,\omega_0,\alpha)$, which are independent in $t$, such that
\begin{equation}\label{krf-alpha}
\int_{X}e^{\alpha(\sup_X\psi-\psi)}\frac{\omega_0^n}{V_{\omega_0}}\le A',\,\,\,\forall \psi\in PSH(X,\omega_t).
\end{equation}
Moreover, we also have a positive number $B=B(X,\omega_0)$ such that for all $t\ge0$,
\begin{equation}
|f_{\omega(t)}|_{L^\infty(X)}\le B, \,\,\,where \,\,\, f_{\omega(t)}:=\frac{\omega(t)^n/V_{\omega(t)}}{\omega_0^n/V_{\omega_0}},
\end{equation}
thanks to \cite[Lemma 4.3]{GS} (or \cite[Lemma 9.2]{GPSS1}). Therefore, we can apply Proposition \ref{prop-vol-explicit} to finish the proof of Theorems \ref{thm-infinite}.
\end{proof}

\begin{proof}[Proof of Theorem \ref{thm-finite}]
The proof is identical as the above proof of Theorem \ref{thm-infinite}, as in the setting of Theorem \ref{thm-finite}, the K\"ahler-Ricci flow $\omega(t)$, $t\in[0,T)$, solving \eqref{KRF2} also satisfies

\begin{equation}
|f_{\omega(t)}|_{L^\infty(X)}\le B, \,\,\,where \,\,\, f_{\omega(t)}:=\frac{\omega(t)^n/V_{\omega(t)}}{\omega_0^n/V_{\omega_0}},
\end{equation}
thanks to \cite[Lemma 7.1]{SW} (or \cite[Lemma 8.1]{GPSS1}). Therefore, we can apply Proposition \ref{prop-vol-explicit} to finish the proof of Theorem \ref{thm-finite}.
\end{proof}

\subsection{A generalization to non-minimal K\"ahler manifolds}\label{tkrf} 
In this subsection, we present a volume noncollapsing property for a twisted K\"ahler-Ricci flow, which contains Theorem \ref{thm-infinite} as a special case. Let $(X,\omega_X)$ be an $n$-dimensional compact K\"ahler manifold, and consider the K\"ahler-Ricci flow starting from an initial K\"ahler metric $\omega_0$ on $X$,
\begin{equation}\label{KRF3}
    \begin{cases}
       \partial_t\omega(t)=
        -\mathrm{Ric}(\omega(t)),\\
        \omega(0)=\omega_0,
    \end{cases}
\end{equation}
which exists a smooth solution for $t\in[0,T)$, where 
\begin{equation}\label{eq-T}
T:=\sup\{t>0|[\omega_0]+t\cdot 2\pi c_1(K_X)>0\}\in(0,+\infty],
\end{equation}
thanks to \cite{C,TZ,Ts}. If $T<+\infty$, namely $X$ is not minimal, it is conjectured by Song-Tian \cite{ST3} that the K\"ahler-Ricci flow $(X,\omega(t))$ should deform $X$ to a compact metric space naturally associated to $X$ and $[\omega_0]$. For the global volume noncollapsing case, important progresses can be found in \cite{SW1,GPSS1}; while the global volume collapsing case keeps largely open in general (see \cite{F,JST,SW2,SSW} for related works when $[\omega_0]$ is rational and of Calabi symmetry). Here we consider an alternative to deforem $X$ to a compact metric space naturally associated to $X$ and $[\omega_0]$, as proposed in \cite[Section 7]{Zys1}. Precisely, we consider the following twisted K\"ahler-Ricci flow starting from $\omega_0$:
\begin{equation}\label{KRF4}
    \begin{cases}
       \partial_t\omega(t)=
        -\mathrm{Ric}(\omega(t))-\omega(t)+\frac{1}{T}\omega_0,\\
        \omega(0)=\omega_0,
    \end{cases}
\end{equation}
which exists a smooth long-time solution for $t\in[0,\infty)$. When $T=\infty$, i.e. $X$ is minimal, the flow \eqref{KRF4} is exactly the one in \eqref{KRF1}. In general, the K\"ahler class $[\omega(t)]=e^{-t}[\omega_0]+(1-e^{-t})(\frac{1}{T}[\omega_0]+2\pi c_1(K_X))$ along \eqref{KRF4} converges to the class $\frac{1}{T}[\omega_0]+2\pi c_1(K_X)$, which is proportional to the limit class $[\omega_0]+T\cdot 2\pi c_1(K_X)$ of the original K\"ahler-Ricci flow \eqref{KRF3} when $T<\infty$, and is $2\pi c_1(K_X)$ when $T=\infty$. 

We now state the main result in this section, which could be regarded as a generalization of Theorem \ref{thm-infinite} on minimal K\"ahler manifolds to not minimal ones.

\begin{theorem}\label{thm-vol-tkrf}
Let $(X,\omega_X)$ be an $n$-dimensional compact K\"ahler manifold, $\omega_0$ a K\"ahler metric on $X$ and $T\in(0,+\infty]$ is defined as in \eqref{eq-T}. Let $\omega(t),\,t\in[0,\infty)$ be the solution to the twisted K\"ahler-Ricci flow \eqref{KRF4}. Then for any integer $k\ge2$ and number $r>n$, there exists a positive number $c=c(X,\omega_0,k,r)$ such that for any $t\in[0,\infty)$, $x\in X$ and $R\in(0,1]$, there holds
\begin{equation}\label{vol-est-tkrf}
\frac{Vol_{\omega(t)}(B_{\omega(t)}(x,R))}{Vol_{\omega(t)}(X)}\ge \frac{cR^{2n}}{\Gamma_{k,r}(-\log R+1)}.
\end{equation}
In particular, for any $\epsilon>0$, there exists a positive number $\tilde c=\tilde c(X,\omega_0,\epsilon)$ such that for any $t\in[0,\infty)$, $x\in X$ and $R\in(0,1]$, there holds
\begin{equation}
\frac{Vol_{\omega(t)}(B_{\omega(t)}(x,R))}{Vol_{\omega(t)}(X)}\ge \frac{\tilde cR^{2n}}{(-\log R+1)^{2n+\epsilon}}.\nonumber
\end{equation}
Consequently, the diameter of $(X,\omega(t))$ is uniformly bounded, and $(X,\omega(t))$ subsequently converges to a compact metric space in Gromov-Hausdorff topology.
\end{theorem}
\begin{proof}
It suffices to check that $f_{\omega(t)}:=\frac{\omega(t)^n/V_{\omega(t)}}{\omega_0^n/V_{\omega_0}}$ is uniformly bounded on $X\times[0,\infty)$. This can be done by the same argument in \cite{GS,GPSS1}. For convenience we contain some details here. Denote the numerical Kodaira dimension of $\frac{1}{T}[\omega_0]+2\pi c_1(K_X)$ by $\kappa$, then one easily has that (see e.g. \cite[Lemma 9.1]{GPSS1}) there is a uniform number $C=C(X,[\omega_0])\ge1$ such that for any $t\ge0$,
\begin{equation}\label{eq-tkrf-1}
C^{-1}e^{-(n-\kappa)t}\le V_{\omega(t)}\le Ce^{-(n-\kappa)t}.
\end{equation}
Fix a smooth representative $\chi$ of $\frac{1}{T}[\omega_0]+2\pi c_1(K_X)$ and then a smooth volume form $\Omega$ on $X$ with $\sqrt{-1}\partial\bar\partial\log\Omega=\chi-\frac{1}{T}\omega_0$ and $\int_X\Omega=1$. Then we can uniquely write $\omega(t)=\omega_t+\sqrt{-1}\partial\bar\partial\varphi(t)$, where $\omega_t:=e^{-t}\omega_0+(1-e^{-t})\chi$ and $\varphi(t)$ solves
\begin{equation}
e^{(n-\kappa)t}(\omega_t+\sqrt{-1}\partial\bar\partial\varphi(t))^n=e^{\partial_t\varphi(t)+\varphi(t)}\Omega,\,\,\,\,\,\varphi(0)=0.
\end{equation}
Then we can identically apply the maximum principle argument in \cite[Lemma 4.3]{GS} (or \cite[Lemma 9.2]{GPSS1}) to prove that $\partial_t\varphi(t)+\varphi(t)$ is uniformly bounded from above on $X\times[0,\infty)$. Therefore, combining with \eqref{eq-tkrf-1}, we see that $f_{\omega(t)}$ is uniformly bounded from above on $X\times[0,\infty)$. Now, as before, we can apply Proposition \ref{prop-vol-explicit} to finish the proof of the volume noncollapsing estimate \eqref{vol-est-tkrf}.

Given \eqref{vol-est-tkrf}, the second part of Theorem \ref{thm-vol-tkrf} immediately follows from Gromov's precompactness (see e.g. \cite[Section 6, Remark]{GPSS1} for details).

Theorem \ref{thm-vol-tkrf}  is proved.
\end{proof}

\section{Applications to K\"ahler family}\label{sect-family}

Let $\mathcal X\to 2\mathbb D$ be a K\"ahler family and $\beta$ a relative K\"ahler form on $\mathcal X$ as defined in subsection \ref{subsect-family}. Fix four numbers $p>1$, $q>n$, $A>0$ and $B>0$, we introduce two sets of relative K\"ahler forms on $\mathcal X$ as follows, according to two different integrability conditions on volume density. For a relative K\"ahler form $\omega$ on $\mathcal X$, set $f_s=f_{\omega_s}:=\frac{\omega_s^n/V_{\omega_s}}{\beta_s^n/V_{\beta_s}}$.
\begin{itemize}
\item[(I)] Set $\mathcal L(\mathcal X,p,A,B)$ denote the set of all relative K\"ahler form $\omega$ on $\mathcal X$ such that for each $s\in2\mathbb D^*$, there hold that 
$$[\omega_s]\le A[\beta_s]\,\,\, in\,\,\, H^{1,1}(X_s,\mathbb R)\,\,\, and \,\,\,\int_{X_s}f_s^p\frac{\beta_s^n}{V_{\beta_s}}\le B.$$

\item[(II)] Set $\mathcal N(\mathcal X,q,A,B)$ denote the set of all relative K\"ahler form $\omega$ on $\mathcal X$ such that for each $s\in2\mathbb D^*$, there hold that 
$$[\omega_s]\le A[\beta_s]\,\,\, in\,\,\, H^{1,1}(X_s,\mathbb R)\,\,\, and \,\,\,\int_{X_s}f_s|\log f_s|^q\frac{\beta_s^n}{V_{\beta_s}}\le B.$$
\end{itemize}

Elementarily, for any $q'>n$, $\mathcal L(\mathcal X,p,A,B)$ is contained in $\mathcal N(\mathcal X,q',A,B')$ for some $B'$. Both the above integrability conditions on volume density are special cases in Kolodziej's work \cite{K} (see the following metric set (III)), and they naturally appeared in various geometric problems, see \cite{DGG,ST3,To} etc..
In general, we also consider the following set of relative K\"ahler forms with a general integrability condition on volume form, which has been intruduced and studied in Kolodziej's work \cite{K} (also see \cite{GGZ,Liuj}):
\begin{itemize}
\item[(III)] For a K-function $\Phi:\mathbb R_+\to\mathbb R_+$, let $\mathcal N(\mathcal X,\Phi,A,B)$ denote the set of all relative K\"ahler form $\omega$ on $\mathcal X$ such that for each $s\in2\mathbb D^*$, there hold that 
$$[\omega_s]\le A[\beta_s]\,\,\, in\,\,\, H^{1,1}(X_s,\mathbb R)\,\,\, and \,\,\,\int_{X_s}\Phi(\log (f_{s}+1))f_s\frac{\beta_s^n}{V_{\beta_s}}\le B.$$

\end{itemize}
For example, choosing $\Phi=\Phi(v)$ to be  $e^{\epsilon v},\,\epsilon>0$, and $|v|^q,\,q>n$, gives the integrability conditions equivalent to those in (I) and (II), respectively.

In the applications of Theorem \ref{Linfty}, we have to verify two assumptions (H1) and (H2). For a fixed manifold case, the assumption (H1) automatically holds if $[\theta]\le C[\omega_X]$ in $H^{1,1}(X,\mathbb R)$, thanks to Tian's $\alpha$-invariant \cite{Ti}. In \cite[Theorem 1.8]{GT}, Guedj-T\^o proved that the assumption (H1) uniformly holds for the relative K\"ahler forms in $\mathcal L(\mathcal X,p,A,B)$ on a K\"ahler family $\mathcal X$. Here, by simply modifying their argument, we present an extension to the set $\mathcal N(\mathcal X,\Phi,A,B)$.
\begin{theorem}\label{thm-alpha}
Fix two numbers $A>0$ and $B>0$, and a K-function $\Phi$. Let $\mathcal N(\mathcal X,\Phi,A,B)$ be as defined as above. There exist two positive numbers $\alpha=\alpha(n,\Phi,A,B)$ and $C=C(n,\Phi,A,B,\alpha)$ such that for each $\omega\in\mathcal N(\mathcal X,\Phi,A,B)$, $s\in\mathbb D^*$ and $\varphi_t\in PSH(X_t,\omega_t)$, there holds
$$\int_{X}e^{\alpha(\sup_X\psi_t-\psi_t)}\frac{\omega_X^n}{V_{\omega_X}}\le C.$$
\end{theorem}
\begin{proof}
The proof is the same as \cite[Theorem 1.8]{GT}, the only difference is that the role of \cite[Theorem 1.9]{DGG} therein is now replaced by the above Theorem \ref{Linfty}.
\end{proof}

As an immediate consequence of Theorems \ref{Linfty} and \ref{thm-alpha}, we also have the following extension of \cite[Theorem 1.9]{GT}.
\begin{theorem}\label{Linfty-family}
Fix two numbers $A>0$ and $B>0$, and a K-function $\Phi$, and let $\mathcal N(\mathcal X,\Phi,A,B)$ be as defined as above and $\omega\in\mathcal N(\mathcal X,\Phi,A,B)$. For arbitrary K-function $\tilde\Phi$ and positive number $\tilde B$, there is a positive number $C=C(n,A,\Phi,B,\tilde\Phi,\tilde B)$ satisfying the following: for any $s\in \mathbb D^*$, if $\psi_s$ is the unique solution to
$$V_{\omega_s}^{-1}(\omega_s+\sqrt{-1}\partial\bar\partial\psi_s)^n=h_s\frac{\beta_s^n}{V_{\beta_s}},\,\,\sup_{X_s}\psi_s=0,$$
with $\int_{X_s}\tilde\Phi(\log(h_s+1))h_s\frac{\beta_s^n}{V_{\beta_s}}\le \tilde B$, then 
$$|\psi_s|_{L^\infty(X_s)}\le C(n,A,\Phi,B,\tilde\Phi,\tilde B).$$
\end{theorem}

Thanks to \cite{GPSS1,GPSS2} (also see \cite{GT,Liuj,NV}), we have the uniform lower bound and $L^1$-bound of Green's function under $L^1(\log L)^q$ condition.

\begin{theorem}\cite{GPSS1,GPSS2}\label{int}
Let $\mathcal X\to 2\mathbb D$ be a K\"ahler family and $\beta$ a relative K\"ahler form on $\mathcal X$ as defined above, and fix three numbers $q>n$, $A>0$ and $B>0$. There exists $C_0=C_0(n,q,A,B)>0$  such that for each $s\in\mathbb D^*$, $x\in X_s$ and $\omega\in\mathcal N(\mathcal X,q,A,B)$, there holds
$$-V_{\omega_s}\inf_{X_s}G_{\omega_s}(x,\cdot)+\int_{X_s}|G_{\omega_s}(x,\cdot)|\omega_s^n\le C_0(n,q,A,B).$$
\end{theorem}
\begin{proof}
Though we are in the K\"{a}hler family setting, one can repeat the arguments in \cite[Lemmas 5.1 and 5.4]{GPSS1} and \cite[Proposition 2.1]{GPSS2} to prove Theorem \ref{int}, given the uniform $L^\infty$ estimates on a K\"{a}hler family in Theorems \ref{Linfty}, \ref{thm-alpha} and \ref{Linfty-family}.
\end{proof}

\begin{proof}[Proof of Theorem \ref{thmgreen-family}]
As we have prepared Theorems \ref{thm-alpha},\ref{Linfty-family} and \ref{int}, we can directly apply Theorem \ref{thmgreen} to conclude Theorem \ref{thmgreen-family}.
\end{proof}

\appendix

\section{A Moser-Trudinger-Onofri type inequality}\label{app-mto}

\begin{theorem}\label{mto}
    Let $\omega$ be a K\"ahler metric $\omega$ satisfying the assumption in Theorem \ref{thmgeom}. For any $r>1$, there exist constants $\theta,C>0$ which depend on $n,\alpha,A,p,K,r$ such that
    \begin{equation}
       \int_X J\left(\frac{\theta|u-\bar{u}|}{|\nabla u|_{L^{2n}(X,\frac{\omega^n}{V_\omega})}}\right) \frac{\omega^n}{V_\omega}\leq C
    \end{equation}
    for all $u\in W^{1,2n}(X)$ where $\bar{u}=\frac{1}{V_\omega}\int_X u\omega^n$ and $J(t)=\exp\{t^{\frac{1}{r}}\}$.
\end{theorem}
\begin{proof}
    Without loss of generality, we may assume $u\in C^1(X)$. Taking $s=\frac{n}{2n-1}$ and from \refe{u(x)} we have
    \begin{equation}\label{mto-u(x)}
        |u(x)-\bar{u}|\leq C\left(\int_X\overline{h}^{2n}(\log(V_\omega\mathcal{G}+1))|\nabla u|^{2n}\frac{\omega^n}{V_\omega}\right)^{\frac{1}{2n}},
    \end{equation}
    where $\overline{h}(t)=t^r$ with $r>1$ is a constant and $\mathcal G=\mathcal G_\omega(x,\cdot)$. Elementarily, $\psi(t)=\log^{2nr}(t+1)$ satisfies $\psi''(t)\leq 0$ when $t$ is larger than a constant $b$. Then by \eqref{mto-u(x)} and Jensen's inequality,
    \begin{equation}\label{u(x)2}
        \begin{aligned}
            &\frac{\theta^{2n}|u(x)-\bar{u}|^{2n}}{\int_X|\nabla u|^{2n}\frac{\omega^n}{V_\omega}}\leq \frac{1}{\int_X|\nabla u|^{2n}\frac{\omega^n}{V_\omega}}\int_X \psi(V_\omega\mathcal{G}+b)|\nabla u|^{2n}\frac{\omega^n}{V_\omega}\\
            &\le \psi\left(\frac{1}{\int_X|\nabla u|^{2n}\frac{\omega^n}{V_\omega}}\int_X V_\omega\mathcal{G}|\nabla u|^{2n}\frac{\omega^n}{V_\omega}+b\right)
        \end{aligned}
    \end{equation}
    where we have chosen $\alpha=1/C$. Taking $\psi^{-1}$ on both sides of \refe{u(x)2} we obtain
    \begin{equation}\label{J(x)}
        J\left(\frac{\theta|u(x)-\bar{u}|}{|\nabla u|_{L^{2n}(X,\frac{\omega^n}{V_\omega})}}\right)\le \frac{1}{\int_X|\nabla u|^{2n}\frac{\omega^n}{V_\omega}}\int_X V_\omega\mathcal{G}|\nabla u|^{2n}\frac{\omega^n}{V_\omega}+b
    \end{equation}
    Integrating \refe{J(x)} over $x\in X$ against $\frac{\omega^n}{V_\omega}$, we have
    \begin{equation}
        \begin{aligned}
            &\int_{x\in X} J\left(\frac{\theta|u(x)-\bar{u}|}{|\nabla u|_{L^{2n}(X,\frac{\omega^n}{V_\omega})}}\right)\frac{\omega^n(x)}{V_\omega}\\
            &\leq\frac{1}{\int_X|\nabla u|^{2n}\frac{\omega^n}{V_\omega}}\int_{x\in X}\int_{y\in X} V_\omega\mathcal{G}(x,y)|\nabla u(y)|^{2n}\frac{\omega^n(y)}{V_\omega}\frac{\omega^n(x)}{V_\omega}+b\\
            &=\frac{1}{\int_X|\nabla u|^{2n}\frac{\omega^n}{V_\omega}}\int_{y\in X}|\nabla u(y)|^{2n}\frac{\omega^n(y)}{V_\omega}\int_{x\in X} \mathcal{G}(x,y) \omega^n(x)+b\\
            &\leq C,
        \end{aligned}
    \end{equation}
    where in the last inequality we have applied Theorem \ref{thmgpss} (0).
    
This completes the proof.
\end{proof}

\section{Proof of Theorem \ref{Linfty}}\label{app-Linfty}
In this subsection, we assume the setting and notations in Theorem \ref{Linfty}. Firstly, thanks to Berman's result \cite[Theorem 1.1]{Be}, we know $u_\theta$ is in $C^{1,\eta}(X,\omega_X)$ for any $\eta\in(0,1)$, and we can fix a sequence of smooth function $u_{\beta}\in PSH(X,\theta),\,\beta\in\mathbb N$, with $|u_\beta-u_\theta|_{C^{1,\eta}(X,\omega_X)}\to0$  as $\beta\to\infty$ and $u_\beta\le1$ on $X$. Also fix a sequence of smooth positive functions $\tau_k:\mathbb R\to\mathbb R$ such that $\tau_k(x)$ decreases to $x\cdot\chi_{\mathbb R_+}(x)$ as $k\to\infty$. Since $\Phi$ is Lipschitz, by Yau's theorem \cite{Y} and the regularity result in \cite[Theorem 1.1]{CH} (also see \cite[Corollary 5]{GPSt}), we fix a $C^{2,\gamma}$ solution $\psi_\Phi$ to 
$$V_\theta^{-1}(\theta+\sqrt{-1}\partial\bar\partial\psi_\Phi)^n=\frac{\Phi(\log (f_\omega+1))}{N_\Phi(\omega)}f_\omega\frac{\omega_X^n}{V_{\omega_X}},\,\,\,\,\,\sup_X\psi_{\Phi}=0,$$
and then define
$$\Psi_1:=u_\beta-\hat\Phi\left(-\frac12\alpha(\psi_\Phi-u_\beta-1)+\Lambda_1\right),$$
where $\hat\Phi(v):=\frac{4N_\Phi(\omega)^{\frac1n}}{\alpha}\int_1^v\Phi(w)^{-\frac1n}dw$, and $\Lambda_1=\Lambda_1(\alpha,\Phi)$ is a fixed positive number with $\frac{\alpha}{2}\hat\Phi'(\Lambda_1)\le1$, e.g. we may define $\Lambda_1$ by
\begin{align}\label{Lambda1}
\Lambda_1:=\max\{\Phi^{-1}(2^nB_\Phi),1\},
\end{align} 
here $\Phi^{-1}$ is understood as the inverse of $\Phi:\mathbb R_+\to\mathbb R_+$.
We may also assume that $\beta$ is large enough with $\psi_\Phi-u_\beta-1\le0$. Also note that $\Psi_1\le u_\beta\le1$.  Directly we have
\begin{align}
&\theta+\sqrt\partial\bar\partial\Psi_1\nonumber\\
&=(1-\frac{\alpha}{2}\hat\Phi')(\theta+\sqrt{-1}\partial\bar\partial u_\beta)+\frac{\alpha}{2}(\hat\Phi')(\theta+\sqrt{-1}\partial\bar\partial\psi_\Phi)-\frac{\alpha^2}{4}(\hat\Phi'')\sqrt{-1}\partial(\psi_\Phi-u_\beta)\wedge\overline\partial(\psi_\Phi-u_\beta)\nonumber\\
&\ge\frac{\alpha}{2}(\hat\Phi')(\theta+\sqrt{-1}\partial\bar\partial\psi_\Phi)\nonumber,
\end{align}
which gives
$$V_\theta^{-1}(\theta+\sqrt\partial\bar\partial\Psi_1)^n\ge2^n\frac{\Phi(\log (f_\omega+1))}{\Phi\left(-\frac12\alpha(\psi_\Phi-u_\beta-1)+\Lambda_1\right)}f_\omega\frac{\omega_X^n}{V_{\omega_X}}$$
and hence
\begin{equation}
f_\omega\frac{\omega_X^n}{V_{\omega_X}}\le\left\{
\begin{aligned}
2^{-n}V_\theta^{-1}(\theta+\sqrt\partial\bar\partial\Psi_1)^n,\,\,\,\,\,\,&if\,\,\,\log (f_\omega+1)\ge-\frac12\alpha(\psi_\Phi-u_\beta-1)+\Lambda_1\nonumber\\
e^{-\frac12\alpha(\psi_\Phi-u_\beta-1)+\Lambda_1}\frac{\omega_X^n}{V_{\omega_X}},\,\,\,\,\,\,&if\,\,\,\log (f_\omega+1)\le-\frac12\alpha(\psi_\Phi-u_\beta-1)+\Lambda_1\nonumber. 
\end{aligned}
\right.
\end{equation}
In particular, we have
\begin{align}\label{volumeformbd}
f_\omega\frac{\omega_X^n}{V_{\omega_X}}\le2^{-n}V_\theta^{-1}(\theta+\sqrt\partial\bar\partial\Psi_1)^n+e^{-\frac12\alpha(\psi_\Phi-u_\beta-1)+\Lambda_1}\frac{\omega_X^n}{V_{\omega_X}}\,\,\,\,\,on\,\,\,\,\,X.
\end{align}
Note that if we set $h_{\Phi,\beta}:=e^{-\frac12\alpha(\psi_\Phi-u_\beta-1)+\Lambda_1}$, then \begin{align}\label{h-int}
\int_Xh_{\Phi,\beta}^2\frac{\omega_X^n}{V_{\omega_X}}\le e^{2\alpha+2\Lambda_1}A_\alpha.
\end{align}
Next, for $w\ge0$ we fix a $C^4$ solution $\psi_{w,k,\beta}$ to the following equation (the right hand side is in $C^{2,\gamma}$):
$$V_\theta^{-1}(\theta+\sqrt{-1}\partial\bar\partial\psi_{w,k,\beta})^n=\frac{\tau_k(-\varphi+\frac12u_\beta+\frac12\Psi_1-w)}{A_{w,k,\beta}}h_{\Phi,\beta}\frac{\omega_X^n}{V_{\omega_X}},\,\,\,\sup_X\psi_{w,k,\beta}=0,$$
with $A_{w,k,\beta}:=\int_X\tau_k(-\varphi+\frac12u_\beta+\frac12\Psi_1-w)h_{\Phi,\beta}\frac{\omega_X^n}{V_{\omega_X}}$. Consider
\begin{align}
\Psi_{w,k,\beta}:=-\epsilon(-\psi_{w,k,\beta}+u_\beta+1+\Lambda_2)^{\frac{n}{n+1}}-(\varphi-\frac12u_\beta-\frac12\Psi_1+w),
\end{align}
where
$$ \epsilon:=\left(\frac{n+1}{n}\right)^{\frac{n}{n+1}}A_{w,k,\beta}^{\frac{1}{n+1}},\,\,\,\,\,\,\Lambda_2:=\frac{2^{n+1}n}{n+1}A_{w,k,\beta}.$$
Then, similar to \cite[page 751-752]{GPTW} or \cite[page 25-26]{Liuj}, we can check that $\Psi_{w,k,\beta}\le0$ on $X$. Actually, let $x_0\in X$ be a point reaching the maximal value of $\Psi_{w,k,\beta}\le0$ and may assume $-(\varphi-\frac12u_\beta-\frac12\Psi_1+w)(x_0)>0$. By direct computation and the maximum principle we have that at $x_0$,

\begin{align}
0\ge\sqrt{-1}\partial\bar\partial\Psi_{w,k,\beta}=&-(\theta+\sqrt{-1}\partial\bar\partial\varphi)+\frac12(\theta+\sqrt{-1}\partial\bar\partial\Psi_1)\nonumber\\
&+\frac{\epsilon n}{n+1}(-\psi_{w,k,\beta}+u_\beta+1+\Lambda_2)^{-\frac{1}{n+1}}(\theta+\sqrt{-1}\partial\bar\partial\psi_{w,k,\beta})\nonumber\\
&+\left(\frac12-\frac{\epsilon n}{n+1}(-\psi_{w,k,\beta}+u_\beta+1+\Lambda_2)^{-\frac{1}{n+1}}\right)(\theta+\sqrt{-1}\partial\bar\partial u_\beta)\nonumber\\
&+\frac{\epsilon n}{(n+1)^2}(-\psi_{w,k,\beta}+u_\beta+1+\Lambda_2)^{-\frac{n+2}{n+1}}\sqrt{-1}\partial(-\psi_{w,k,\beta}+u_\beta)\wedge\bar\partial(-\psi_{w,k,\beta}+u_\beta).
\end{align} 
Since $\frac12-\frac{\epsilon n}{n+1}(-\psi_{w,k,\beta}+u_\beta+1+\Lambda_2)^{-\frac{1}{n+1}}\ge\frac12-\frac{\epsilon n}{n+1}\Lambda_2^{-\frac{1}{n+1}}=0$ and $\sqrt{-1}\partial(-\psi_{w,k,\beta}+u_\beta)\wedge\bar\partial(-\psi_{w,k,\beta}+u_\beta)\ge0$, we conclude that
\begin{align}
(\theta+\sqrt{-1}\partial\bar\partial\varphi)\ge\frac12(\theta+\sqrt{-1}\partial\bar\partial\Psi_1)+
\frac{\epsilon n}{n+1}(-\psi_{w,k,\beta}+u_\beta+1+\Lambda_2)^{-\frac{1}{n+1}}(\theta+\sqrt{-1}\partial\bar\partial\psi_{w,k,\beta})\nonumber,
\end{align}
and hence
\begin{align}\label{volumeformlower} 
&V_\theta^{-1}(\theta+\sqrt{-1}\partial\bar\partial\varphi)^n\nonumber\\
&\ge \frac{1}{2^{n}V_\theta}(\theta+\sqrt{-1}\partial\bar\partial\Psi_1)^n+\left(\frac{\epsilon n}{n+1}(-\psi_{w,k,\beta}+u_\beta+1+\Lambda_2)^{-\frac{1}{n+1}}\right)^nV_\theta^{-1}(\theta+\sqrt{-1}\partial\bar\partial\psi_{w,k,\beta})^n\nonumber\\
&= \frac{1}{2^{n}V_\theta}(\theta+\sqrt{-1}\partial\bar\partial\Psi_1)^n+\left(\frac{\epsilon n}{n+1}(-\psi_{w,k,\beta}+u_\beta+1+\Lambda_2)^{-\frac{1}{n+1}}\right)^n\frac{\tau_k(-\varphi+\frac12u_\beta+\frac12\Psi_1-w)}{A_{w,k,\beta}}h_{\Phi,\beta}\frac{\omega_X^n}{V_{\omega_X}}\nonumber\\
&\ge \frac{1}{2^{n}V_\theta}(\theta+\sqrt{-1}\partial\bar\partial\Psi_1)^n+\left(\frac{\epsilon n}{n+1}(-\psi_{w,k,\beta}+u_\beta+1+\Lambda_2)^{-\frac{1}{n+1}}\right)^n\frac{-\varphi+\frac12u_\beta+\frac12\Psi_1-w}{A_{w,k,\beta}}h_{\Phi,\beta}\frac{\omega_X^n}{V_{\omega_X}}.
\end{align}
On the other hand, recall that by  \eqref{volumeformbd} we have
\begin{align}\label{volumeformbd'} 
V_\theta^{-1}(\theta+\sqrt{-1}\partial\bar\partial\varphi)^n=f_\omega\frac{\omega_X^n}{V_{\omega_X}}\le2^{-n}V_\theta^{-1}(\theta+\sqrt\partial\bar\partial\Psi_1)^n+h_{\Phi,\beta}\frac{\omega_X^n}{V_{\omega_X}}\,\,\,\,\,on\,\,\,\,\,X.
\end{align}
Now combining \eqref{volumeformlower} and \eqref{volumeformbd'} gives that
$$\left(\frac{\epsilon n}{n+1}(-\psi_{w,k,\beta}+u_\beta+1+\Lambda_2)^{-\frac{1}{n+1}}\right)^n\frac{-\varphi+\frac12u_\beta+\frac12\Psi_1-w}{A_{w,k,\beta}}\le1\,\,\,\,\,at\,\,\,\,\,x_0,$$
which, using the definitions of $\epsilon$, is equivalent to $\Psi_{w,k,\beta}(x_0)\le0$. 
In conclusion, we have checked that $\Psi_{w,k,\beta}\le0$ on $X$. Now we set
$$\phi(w):=\int_{\Omega_w}h_{\Phi,\beta}\frac{\omega_X^n}{V_{\omega_X}}, \,\,\,\,\,\Omega_w:=\{-\varphi+\frac12u_\beta+\frac12\Psi_1-w>0\}.$$
Observe that for any $v>0$
\begin{align}\label{ite-1}
v\phi(w+v)&=v\int_{\Omega_{w+v}}h_{\Phi,\beta}\frac{\omega_X^n}{V_{\omega_X}}\le\int_{\Omega_{w+v}}(-\varphi+\frac12u_\beta+\frac12\Psi_1-w)h_{\Phi,\beta}\frac{\omega_X^n}{V_{\omega_X}}\nonumber\\
&\le\int_{\Omega_{w}}(-\varphi+\frac12u_\beta+\frac12\Psi_1-w)h_{\Phi,\beta}\frac{\omega_X^n}{V_{\omega_X}}=A_{w,\beta}=\lim_{k\to\infty}A_{w,k,\beta}.
\end{align}
Using $\Psi_{w,k,\beta}\le0$ on $X$, we estimate, for any $k\ge1$,
\begin{align}
A_{w,\beta}&=\int_{\Omega_{w}}(-\varphi+\frac12u_\beta+\frac12\Psi_1-w)h_{\Phi,\beta}\frac{\omega_X^n}{V_{\omega_X}}\nonumber\\
&\le\epsilon\int_{\Omega_{w}}(-\psi_{w,k,\beta}+u_\beta+1+\Lambda_2)^{\frac{n}{n+1}}h_{\Phi,\beta}\frac{\omega_X^n}{V_{\omega_X}}\nonumber\\
&\le\epsilon\int_{\Omega_{w}}(-\psi_{w,k,\beta}+2+\Lambda_2)^{\frac{n}{n+1}}h_{\Phi,\beta}\frac{\omega_X^n}{V_{\omega_X}}\nonumber\\
&\le\epsilon\left(\int_{\Omega_{w}}(-\psi_{w,k,\beta}+2+\Lambda_2)^{\frac{2n^2}{n+1}}h_{\Phi,\beta}\frac{\omega_X^n}{V_{\omega_X}}\right)^{\frac{1}{2n}}\left(\int_{\Omega_{w}}h_{\Phi,\beta}\frac{\omega_X^n}{V_{\omega_X}}\right)^{1-\frac{1}{2n}}\nonumber\\
&\le\epsilon (b_{\alpha})^{\frac{1}{2n}}\left(\int_{\Omega_{w}}e^{\frac{\alpha}{2}(-\psi_{w,k,\beta}+2+\Lambda_2)}h_{\Phi,\beta}\frac{\omega_X^n}{V_{\omega_X}}\right)^{\frac{1}{2n}}\left(\phi(w)\right)^{1-\frac{1}{2n}}\nonumber\\
&\le\epsilon(b_{\alpha})^{\frac{1}{2n}}e^{\frac{\alpha(2+\Lambda_2)}{4n}}\left(\int_{\Omega_{w}}e^{-\alpha\psi_{w,k,\beta}}\frac{\omega_X^n}{V_{\omega_X}}\right)^{\frac{1}{4n}}\left(\int_{\Omega_{w}}h_{\Phi,\beta}^2\frac{\omega_X^n}{V_{\omega_X}}\right)^{\frac{1}{4n}}\left(\phi(w)\right)^{1-\frac{1}{2n}}\nonumber\\
&\le C^*A^{\frac{1}{n+1}}_{w,k,\beta}\left(\phi(w)\right)^{1-\frac{1}{2n}},
\end{align}
where $C^*:=(b_{\alpha})^{\frac{1}{2n}}\left(\frac{n+1}{n}\right)^{\frac{n}{n+1}}e^{\frac{4\alpha+2\Lambda_1+\alpha\Lambda_2}{4n}}A_\alpha^{\frac{1}{2n}}$, $b_\alpha=b(n,\alpha)$ is a positive number satisfying \begin{align}\label{balpha}
v^{\frac{2n^2}{n+1}}\le b_\alpha e^{\frac{\alpha}{2}v},\,\,\,\,\,\forall v\ge1,
\end{align} 
and we have used the expression of $\epsilon$, and \eqref{h-int}. Letting $k\to\infty$, by a similar argument we have
\begin{align}
\lim_{k\to\infty}\Lambda_2&=\frac{2^{n+1}n}{n+1}\lim_{k\to\infty}A_{w,k,\beta}\nonumber\\
&=\int_{\Omega_{w}}(-\varphi+\frac12u_\beta+\frac12\Psi_1-w)h_{\Phi,\beta}\frac{\omega_X^n}{V_{\omega_X}}\nonumber\\
&\le\int_{X}(-\varphi+1)h_{\Phi,\beta}\frac{\omega_X^n}{V_{\omega_X}}\nonumber\\
&\le b_\alpha\int_{X}e^{\frac{\alpha}{2}(-\varphi+1)}h_{\Phi,\beta}\frac{\omega_X^n}{V_{\omega_X}}\nonumber\\
&\le b_\alpha e^{\frac{\alpha}{2}}\left(\int_{X}e^{-\alpha\varphi}\frac{\omega_X^n}{V_{\omega_X}}\right)^{\frac12}\left(\int_{X}h_{\Phi,\beta}^2\frac{\omega_X^n}{V_{\omega_X}}\right)^{\frac12}\nonumber\\
&\le b_\alpha e^{\frac{3\alpha}{2}+\Lambda_1}A_\alpha\nonumber.
\end{align}
Now we set 
\begin{align}\label{C0}
C_0:=(b_{\alpha})^{\frac{1}{2n}}\left(\frac{n+1}{n}\right)^{\frac{n}{n+1}}e^{\frac{4\alpha+2\Lambda_1+\alpha b_\alpha A_\alpha e^{\frac{3\alpha}{2}+\Lambda_1}}{4n}}A_\alpha^{\frac{1}{2n}},
\end{align}
and arrive at
$$A_{w,\beta}\le (C_0)^{\frac{n+1}{n}}\phi(w)^{1+\frac{n-1}{2n^2}},$$
plugging which into \eqref{ite-1} gives
\begin{align}
v\phi(w+v)\le (C_0)^{\frac{n+1}{n}}\phi(w)^{1+\frac{n-1}{2n^2}}.
\end{align}
Observe that 
$$\phi(w)=\int_{\Omega_w}h_{\Phi,\beta}\frac{\omega_X^n}{V_{\omega_X}}\le w^{-1}\int_{\Omega_w}(-\varphi+1)h_{\Phi,\beta}\frac{\omega_X^n}{V_{\omega_X}}\le\frac{b_\alpha e^{\frac{3\alpha}{2}+\Lambda_1}A_\alpha}{w}.$$
By a classic iteration of De Giorgi (see \cite[Lemma 2]{GPT}), if we define $w_0$ and $w_1$ by
\begin{align}
w_0:=b_\alpha e^{\frac{3\alpha}{2}+\Lambda_1}A_\alpha(2C_0)^{\frac{2n^2}{n-1}}\nonumber,
\end{align}
and
\begin{align}
w_1:=w_0+\frac{1}{1-2^{-\frac{n-1}{2n^2}}}=b_\alpha e^{\frac{3\alpha}{2}+\Lambda_1}A_\alpha(2C_0)^{\frac{2n^2}{n-1}}+(1-2^{-\frac{n-1}{2n^2}})^{-1},
\end{align}
then we have $\phi(w)=0$ for any $w\ge w_1$, namely, $-\varphi+\frac12u_\beta+\frac12\Psi_1\le w_1$ on $X$. Note  that 
$$-\varphi+\frac12u_\beta+\frac12\Psi_1=-\varphi+u_\beta-\frac12\hat\Phi\left(-\frac12\alpha(\psi_\Phi-u_\beta-1)+\Lambda_1\right)\ge-\varphi+u_\beta-4\alpha^{-1}N_\Phi(\omega)^{\frac1n}C_\Phi,$$
therefore,
\begin{align}
-\varphi+u_\beta\le w_1+4\alpha^{-1}N_\Phi(\omega)^{\frac1n}C_\Phi=b_\alpha e^{\frac{3\alpha}{2}+\Lambda_1}A_\alpha(2C_0)^{\frac{2n^2}{n-1}}+(1-2^{-\frac{n-1}{2n^2}})^{-1}+4\alpha^{-1}N_\Phi(\omega)^{\frac1n}C_\Phi.
\end{align}
Letting $\beta\to\infty$ and setting 
\begin{align}\label{C}
C(n,\alpha,A_\alpha,\Phi,B_\Phi,C_\Phi):=b_\alpha e^{\frac{3\alpha}{2}+\Lambda_1}A_\alpha(2C_0)^{\frac{2n^2}{n-1}}+(1-2^{-\frac{n-1}{2n^2}})^{-1}+4\alpha^{-1}B_\Phi^{\frac1n}C_\Phi,
\end{align}
where $C_0$, $b_\alpha$ and $\Lambda_1$ are defined in \eqref{C0}, \eqref{balpha} and \eqref{Lambda1}, respectively, 
we have completed the proof of Theorem \ref{Linfty}.

\section*{Acknowledgements}
The second-named author is grateful to Zhenlei Zhang for teaching an inspiring summer mini-course on complex Monge-Amp\`ere equations at Hunan University in 2023 and valuable conversations. Both authors thank Huai-Dong Cao, Wangjian Jian, Chengjie Yu and Zhenlei Zhang for their comments on the results. This work is partially supported by National Natural Science Foundation of China (No. 12371057), Natural Science Foundation of Hunan Province (No. 2024JJ2006) and The Science and Technology Innovation Program of Hunan Province (No. 2023RC3096).

\end{document}